\documentclass[11pt, bibliography=totoc] {scrartcl}
\usepackage[ansinew]{inputenc}
\usepackage[T1]{fontenc}
\usepackage{lmodern}
\usepackage{amsfonts}   
\usepackage{amssymb}
\usepackage{amsmath}
\usepackage{mathrsfs}
\usepackage{amsthm}
\usepackage {mathbbol}
\usepackage[arrow, matrix, curve]{xy}
\usepackage{graphicx}
\usepackage{setspace}
\usepackage{color}
\usepackage{tikz}
\usetikzlibrary{matrix}
\usepackage{cite}

\newtheoremstyle{style1}
	   {15pt}       
	    {25pt}      
	    {\itshape}  
	    {}          
	    {\bfseries} 
	    {:}         
	    {.5em}      
	    {}          

\newtheoremstyle{style2}
 {15pt}                   
 {25pt }                   
 {}                   
 {}                      
 {\bfseries}  
 {:}                     
 {.5em}              
 {}

\theoremstyle{style1} 
\newtheorem{theorem}{Theorem} \numberwithin{theorem}{section}
\newtheorem{lemma}[theorem]{Lemma} 
\newtheorem{proposition}[theorem]{Proposition}
\newtheorem{corollary}[theorem]{Corollary}
\newtheorem{assumption}[theorem]{Assumption}

\theoremstyle{style2}
\newtheorem{definition}[theorem]{Definition} 

\newtheorem{remark}[theorem]{Remark}
\newtheorem{construction}[theorem]{Construction}

\renewenvironment{proof}{\vspace{0pt}\noindent{\it Proof:}\hspace{0.2em}}{\hfill\qed\vspace{20pt}}

\usepackage[headsepline, footsepline, plainfootsepline]{scrpage2}
\ohead{\leftmark}
\automark[section]{chapter}

\setkomafont{sectioning}{\normalfont\normalcolor\bfseries}

\begin{document}
 
\title {Stokes Structure and Direct Image of Irregular Singular $\mathcal D$-Modules\\}
\vskip 2cm

\author{Hedwig Heizinger }
\date{}

\maketitle

\thispagestyle{empty}

\setlength\parindent{0pt}

\begin{abstract}
\emph{Abstract:} In this paper we will present a way of examining the Stokes structure of certain irregular singular $\mathcal D$-modules, namely the direct image of exponentially twisted regular singular meromorphic connections, in a topological point of view. This topological description enables us to compute Stokes data for an explicit example concretely.
\end{abstract}

\begin{center}
\noindent\rule{4cm}{0.4pt}
\end{center}


\section{Introduction -- Preliminaries}

Recently there was a lot of progress in proving an irregular Riemann-Hilbert correspondence in higher dimensions. Sabbah introduced the notion of \emph{good} meromorphic connections and proved an equivalence of categories in this case \cite{Sab}, which was extended to the general case by Mochizuki \cite{Moc09} and Kedlaya \cite{Ked10}. Furthermore d'Agnolo/Kashiwara proved an irregular Riemann-Hilbert correspondence for all dimensions using subanalytic sheaves \cite{DAK}.\\
Nevertheless it is still difficult to describe the Stokes phenomenon for explicit situations and to calculate Stokes data concretely.  In their recent article, Hien/Sabbah developed a topological way to determine Stokes data of the Laplace transform of an elementary meromorphic connection\cite{HiS14}. The techniques used by Hien/Sabbah can be adapted to other situations. Hence in this article we will present a topological view of the Stokes phenomenon for the direct image of an exponentially twisted meromorphic connection $\mathcal M$ in a 2-dimensional complex manifold. Namely we will consider the following situation: \\
Let $X= \Delta \times \mathbb P^1$ be a complex manifold, where $\Delta$ denotes an open disc in $0\in \mathbb C$ with coordinate $t$. We denote the coordinate of $\mathbb P^1$ in $0$ by $x$ and the coordinate in $\infty$ by $y= \frac 1 x$.   Let  $\mathcal M$ be a regular singular holonomic $\mathcal D_X$-module. We have the following projections: 
	\[
	\begin{xy}
  		\xymatrix
  		{
  		&  \Delta \times \mathbb P^1  \ar[dl]_{p} \ar[rd]^{q}& \\
  		\Delta & & \mathbb P^1
 	 	}
	\end{xy}
	\]
	
Let $\overline D$ denote the singular locus of $\mathcal M$, which consists of  $ \{ 0\} \times \mathbb P^1 =: D$, $\Delta \times \{ \infty\}$ and some additional components.  We will distinguish between the components 
	\begin{itemize}
		\item $S_{i \in I}$ $(I=\{ 1, \ldots, n\})$,  which meet  $D$ in the point $ \left( 0,\infty\right)$ and 
		\item $\widetilde S_{j\in J}$ $( J = \{1, \ldots, m\})$, which meet  $\{0\} \times \mathbb P^1$ in some other point.
	\end{itemize}
Furthermore we will require the following conditions on $\overline D$:

\begin{assumption}  \label {Ass 1} Locally in $ \left( 0,\infty\right)$ the irreducible components $S_i$ of the divisor $\overline D$ achieve the following conditions: 
	\begin{itemize}
 	 	\item $S_i: \mu_i \left( t\right) y= t^{q_i}$, where $\mu_i$ is holomorphic and $\mu_i \left( 0\right) \neq0$.
 	 	\item For  $i \neq j$ either $q_i \neq q_j$ or $\mu_i \left( 0\right) \neq \mu_j \left( 0\right)$  holds. 
 	\end{itemize}
\end{assumption}	
	
\begin{assumption} \label {Ass 2} The irreducible components $\widetilde S_j$ intersect $D$  in pairwise distinct points. Moreover we assume $\widetilde S_j$ to be smooth, i.\:e. locally around the intersection point they can be described as 
\[\widetilde S_j: \mu_j \left( t\right) x= t^{q_j}.\]
\end{assumption}	

\begin{center}
	\begin{tikzpicture}[domain=-1.3:1.3] 
	\draw[-]  ( 0,-4.2) --  ( 0,0.2) node[above] {$( 0,\infty)$};
	\draw[-] [color= red] ( -2,-4) -- ( 2,-4) node[right] {$\Delta \times \{0\}= \widetilde S_1$};
	\draw[-]  ( -2,0) -- ( 2,0) node[right]{$\Delta \times \{ \infty\}$};
	\draw[color=blue] plot ( \x,{ -\x *\x *1.2}) node[right] {$S_2$};
  	\draw[color=blue] plot ( \x,{-\x *\x *\x* 0.6}) node[right] {$S_1$};
 	 \draw[color= red] plot ( \x, {-\x *\x* 0.1 -2.3}) node [right] {$\widetilde S_3$};
 	 \draw[color= red] plot ( \x , {\x *\x* \x* 0.2 -3.4 }) node [right] {$\widetilde S_2$};
	\end{tikzpicture}
\end{center}
\vspace*{0.5cm}
We want to examine  $p_+ \left(  \mathcal M \otimes \mathcal E^{q}\right)$. This is a complex with  $\mathcal H^k p_+ \left(  \mathcal M \otimes \mathcal E^{q}\right)= 0 $ for $k \neq -1,0$. Furthermore one can show that even $\mathcal H^{-1} p_+ \left(\mathcal M \otimes \mathcal E^{q}\right)= 0 $ on $\Delta^{\ast}$ (cf. \cite{Sab2}, p. 161), i.\:e. it is only supported in $0$. Therefore, we will consider $\mathcal H^0 p_+ \left(  \mathcal M \otimes \mathcal E^{q}\right)$. 
We will assume $\Delta$ small enough such that $0$ is the only singularity of  the $\mathcal D_{\Delta}$-module  $\mathcal H^0p_+ \left(  \mathcal M \otimes \mathcal E^{q}\right)$ and we denote its germ at $0$ by
	$$ \mathcal N:= \left( \mathcal H^0p_+ \left(  \mathcal M \otimes \mathcal E^{q}\right)\right)_0.$$
	
In the following we will take a closer look at the Stokes-filtered local system $(\mathcal L, \mathcal L_{\leq \psi})$ (Chapter 2), which is associated to the $\mathcal D_{\Delta}$-module $\mathcal N$ by the irregular Riemann-Hilbert correspondence as mentioned above. We will use an isomorphism
\[ \Omega: \mathcal L_{\leq \psi} \xrightarrow{\cong} \mathcal H^1 R \widetilde p_{\ast} \operatorname{DR}^{mod \: D}  \left( \mathcal M \otimes  \mathcal E^{\frac 1 y}\otimes \mathcal E^{-\psi}\right)\]
(proved by Mochizuki) to develop a topological description for $\mathcal L_{\leq \psi}$. \\

In Chapter 3 we will use this topological perspective to present a way of determining Stokes matrices for an explicit example, where the singular locus of a meromorphic connection $\mathcal M$ of rank $r$  only consists of two additional irreducible components, namely $(S_1: y=t)$ and $(\widetilde S_1: x=0$).  We will describe the Stokes-filtered local system $\mathcal L$ (which in this case will be of exponential type) in terms of linear data, namely a set of linear Stokes data of exponential type, defined as follows:\\

Let $ \Phi =\{ \phi_i \mid i \in I\}$ denote a finite set of exponents $\phi_i$ of pole order $\leq 1$ and let $\theta_0 \in \mathbb S^1$ be a generic angle, i.\:e. it is no Stokes direction with respect to the $\phi_i$s. We  get a unique ordering of the exponents $\phi_0 <_{\theta_0} \phi_1  <_{\theta_0} \ldots  <_{\theta_0} \phi_n$ and the reversed ordering for $\theta_1 := \theta_0 + \pi$. 

\begin{definition}[\cite{HS}, Def 2.6]  \label{SD} The \emph {category  of Stokes data of exponential type} (for a set of exponents $\Phi$ ordered by $\theta_0$) has objects consisting of two families of $\mathbb C$-vector spaces $ \left( G_{\phi_i}, H_{\phi_i}\right)$
 and two morphisms 

	\[
	\begin{xy}
  		\xymatrix
  		{
  	\bigoplus\limits_{i=0}^{n} G_{\phi_i} \ar[r]^{S}& \bigoplus\limits_{i=0}^{n} H_{\phi_i} & \bigoplus\limits_{i=0}^{n} H_{\phi_{n-i}} \ar[r]^{S'}& \bigoplus\limits_{i=0}^{n} G_{\phi_{n-i}}
 	 	}
	\end{xy}
	\]
such that
	\begin{enumerate}
		\item $S$ is a block upper triangular matrix, i. e. $S_{ij}: G_{\phi_i} \to H_{\phi_j}$ is zero for $i >j $ and $S_{ii}$ is invertible (thus $S$ is invertible and $						\operatorname{dim} \: G_{\phi_i} = \operatorname{dim}\: H_{\phi_i}$)
		\item $S'$ is a block lower triangular matrix, i. e. $S'_{ij}: H_{\phi_{n-i}} \to G_{\phi_{n-j}}$ is zero for $i <j $ and $S'_{ii}$ is invertible (thus $S'$ is invertible)
	\end{enumerate}
A morphism consists of morphisms of $\mathbb C$-vector spaces $\lambda^G_{i}: G_{\phi_i} \to G'_{\phi_i}$ and  $\lambda^H_{i}: H_{\phi_i} \to H'_{\phi_i}$, which are compatible with the corresponding diagrams.
\end{definition}
The correspondence between Stokes-filtered local systems  of exponential type and linear Stokes data is stated in the following theorem. For a proof we refer to \cite{HS}, p. 12/13.

\begin{theorem} There is an equivalence of categories between the Stokes-filtered local systems of exponential type and  Stokes data of exponential type.
\end{theorem}

By associating the Stokes-filtered local system to a set of Stokes data via this equivalence of categories and using the isomorphism $\Omega$ we will finally get an explicit description of the Stokes data in our concrete example. It is stated in the following 

 \begin{theorem} \label{theorem2} Fix the following data: 
	 \begin{itemize}
		 \item $L_0= \mathbb V \oplus \mathbb V$, $L_1= \mathbb V \oplus \mathbb V$
		 \item  $S_0^1= N_{\pi}= \begin{pmatrix} -1 & 1-ST^{-1} \\ 0 & -ST^{-1} \end{pmatrix}$,  $S_1^0= \left(  \mu_0^{\pi} \circ \mu_{\pi}^0\right) \cdot N_0= \begin{pmatrix} U & 0 \\ 				0 & U \end{pmatrix} \cdot \begin{pmatrix} -TS^{-1} & 0 \\ 1-TS^{-1} & -1 \end{pmatrix}$
	 \end{itemize}
where $\mathbb V$ is the generic stalk of the local system attached to $\mathcal M$ and $S$,  $T$ denote the monodromies around the strict transforms of  the irreducible components $\widetilde S_1$, $S_1$ in the singular locus of $\mathcal M$ and $U$ denotes the mondromy around the component $ \{0\} \times \mathbb P^1$. Then 
 		$$ \left( L_0, L_1, S_0^1, S_1^0\right)$$
 defines a set of Stokes data for $\mathcal H^0 p_+ \left( \mathcal M \otimes \mathcal E^{\frac 1 y}\right)$.
 \end{theorem}


\section{Stokes-filtered local system $ \left( \mathcal L, \mathcal L_{\leq \psi}\right)$}

In \cite{Rou} Roucairol examined the formal decomposition of direct images of $\mathcal D$-modules in the situation presented above.  She determined the exponential factors as well as the rank of the corresponding regular parts appearing in the formal decomposition of $\mathcal N$  . In our case this leads to the following  result: 

\begin {theorem}[\cite{Rou}, Thm 1] \label{form_dec} Let $\mathcal M$ be a regular singular $\mathcal D_X$-module with singular locus $\overline D$ which achieves the previous assumptions  \ref {Ass 1}  and \ref {Ass 2}. Then $ \hat{\mathcal N}:= \left(  \mathcal H^0p_+ \left(  \mathcal M \otimes \mathcal E^q\right)\right)_0^{\land}$ decomposes as 
		$$ \hat{\mathcal N}= R_0 \oplus \bigoplus_{i \in I} \left( R_i \otimes \mathcal E^{\psi_i \left( t\right)}\right) ,$$
where $R_0,  R_i \:  \left( i \in I\right)$ are regular singular $\mathcal D_{\Delta}$- modules and $\psi_i \left( t\right)= \mu_i \left( t\right) t^{-q_i}$. Moreover we have: 
	\begin{itemize}
		\item $rk  \left( R_i\right) = \operatorname{dim} \: \Phi_{P_i}$ where $\Phi_{P_i}$ denotes the vanishing cycles of $\operatorname{DR} (e^+ \mathcal M)$ at the intersection 			point $P_i$ of (a strict transform of) $S_i$ with the exceptional divisor after a suitable blow up $e$.
		\item $rk \left( R_0\right)= \sum_{j \in J} \operatorname{dim} \:  { \Phi_{\widetilde P_j}}$, where $\Phi_{\widetilde P_j}$ denotes the vanishing cycles of $\operatorname{DR} 			(\mathcal M)$ at the intersection point $\widetilde P_j$ of $\widetilde S_j$ with $D$.
	\end{itemize}
\end{theorem}

\begin{proof}
The formal decomposition and the statement about the rank of the $R_i$ is exactly Roucairol's theorem applied to our given situation. With the same arguments as in Roucairol's proof one can also show that $rk \left( R_0\right)= \sum_{j \in J} \operatorname{dim} \:  { \Phi_{\widetilde P_j}}$.
\end{proof}

Identify $ \mathbb S^1= \left\{ \vartheta \mid \vartheta \in [0, 2\pi) \right\}$ and denote $\mathcal P:= x^{-1} \mathbb C[x^{-1}]$. Let us recall the following definition of a Stokes-filtered local system.
\begin{definition}[\cite{Sab}, Lemma 2.7] \label{SFLS}
Let $\mathcal L$ be a local system of $\mathbb C$-vector spaces on $\mathbb  S^1$. We will call $ \left( \mathcal L, \mathcal L_{\leq}\right)$ a \emph{Stokes-filtered local system}, if it is equipped with a family of subsheaves $\mathcal L_{\leq \phi}$ (indexed by $ \phi \in \mathcal P:= x^{-1} \mathbb C[x^{-1}]$) satisfying the following conditions: 
	\begin{enumerate}
		\item For all $\vartheta \in  \mathbb S^1$, the germs $\mathcal L_{\leq \phi, \vartheta}$ form an exhaustive increasing filtration of $\mathcal L_{\vartheta}$
		\item $gr_{\phi} \mathcal L:= \mathcal L_{\leq \phi} / \mathcal L_{< \phi}$ is a local system on $ \mathbb S^1$ (where $\mathcal L_{< \phi, \vartheta}:= \sum_{\psi 					<_{\vartheta}\phi} \mathcal L_{\leq \psi, \vartheta}$)
		\item $\operatorname{dim} \: \mathcal L_{\leq \phi, \vartheta}= \sum_{\psi \leq_{\vartheta} \phi} \operatorname{dim} \: gr_{\psi} \mathcal L_{\vartheta}$ 
	\end{enumerate}
\end{definition}

Assume that $\mathcal M$ (and therefore the direct image $\mathcal N$ defined above) is a meromorphic connection. Let $\mathcal L'$ denote the local system on $\Delta^{\ast}$ corresponding to the meromorphic connection $\mathcal N$. Moreover let $\pi: \widetilde{\Delta} \to \Delta, (r, e^{i\vartheta}) \mapsto r \cdot e ^{i \vartheta}$ be the real oriented blow up of $\Delta$ in the singularity $0$ and $ j: \Delta^{\ast} \hookrightarrow \widetilde \Delta$. Consider $j_{\ast} \mathcal L'$ and restrict it to the boundary $\partial \Delta$. We get a local system on $\mathbb S^1 \cong \partial \widetilde{\Delta}$ and define  $\mathcal L := j_{\ast} \mathcal L'_{\mid \mathbb S^1}$.
 For $\phi \in\mathcal P$ define
 \[ \mathcal L_{\leq \phi} := \mathcal H^0 \: \operatorname{DR}^{mod \: 0}_{\partial \widetilde {\Delta}} (\mathcal N \otimes \mathcal E^{-\phi}) \:  \text{ and } \:  \mathcal L_{<\phi} := \mathcal H^0 \: \operatorname{DR}^{< 0}_{\partial \widetilde {\Delta}} (\mathcal N \otimes \mathcal E^{-\phi}),\]
where $\operatorname{DR}^{mod \: 0}$ resp. $\operatorname{DR}^{< 0}$ denotes the moderate (resp.  rapid decay) de Rham complex. 
According to the equivalence of categories between germs of a $\mathcal O_{\Delta}(\ast 0)$-connection and Stokes-filtered local systems stated by Deligne/Malgrange (cf. \cite{Mal91}),  $(\mathcal L, \mathcal L_{\leq})$ forms the Stokes-filtered local system associated to $\mathcal N$. \\
 Moreover by the Hukuhara-Turrittin-Theorem (cf. \cite{Sab07}, p. 109) the formal decomposition of $\mathcal N$ can be lifted locally on sectors to  $\partial \widetilde{\Delta}=\mathbb S^1$. Thus to determine the filtration $\mathcal L_{\leq}$,  it is enough to consider the set of exponential factors (respectively their polar part) 
 appearing in the formal decomposition, since the moderate growth property of the solutions of an elementary connection $\mathcal R \otimes \mathcal E^{\phi}$ ($\phi \in \mathcal P$) only depends on the asymptotical behavior of $e^{\phi}$. We denote the set of exponential factors of  the formal decomposition of $\mathcal N$ by $\mathcal P_{\mathcal N} := \{ \psi_i  \mid i \in I_0\}$ whereby $I_0:= I \cup \{ 0\}$ and $\psi_0 := 0$. 

\begin {definition}
For $\vartheta \in \mathbb S^1$ we define the following ordering on $\mathcal P$: $\phi \leq_{\vartheta} \psi :\Leftrightarrow e^{\phi-\psi} \in \mathcal A^{mod \: 0}$.
\end{definition}

\begin{remark} \label{remark1} 
	\begin{enumerate}
		\item For $ \psi_i, \psi_j \in \mathcal P_{\mathcal N}$ appearing in the formal decomposition of $\mathcal N$ we can determine $\vartheta \in \mathbb S^1$, such that $				\psi_i \leq_{\vartheta} \psi_j$ (cf.\cite{Sab}, ex. 1.6): 
			\[ \psi_i \leq_{\vartheta} \psi_j \Leftrightarrow \psi_i - \psi_j \leq_{\vartheta} 0 \Leftrightarrow \mu_i \left( t\right) t^{-q_i} - \mu_j \left( t\right) t^{-q_j} \leq_{\vartheta} 0. \] 
			There exists a finite set of $\vartheta \in \mathbb S^1$, where $\psi_i$ and $\psi_j$ are not comparable, i.\,e. neither $ \psi_i \leq_{\vartheta} \psi_j$ nor $\psi_j 				\leq_{\vartheta} \psi_i$ holds. We call these angles the \emph{Stokes directions} of $(\psi_i, \psi_j)$. 
		\item For $\vartheta_0 \in \mathbb S^1$ not being a Stokes direction of any pair $ \left( \psi_i, \psi_j\right)$ in $\mathcal P_{\mathcal N}$ we get a total ordering of the $				\psi_i$s with respect to $\vartheta_0$: 
			$\psi_{\tilde0} <_{\vartheta_0} \ldots <_{\vartheta_0} \psi_{\tilde n}.$
	\end{enumerate}
\end{remark}

\begin{corollary} \label{cor1} For $\vartheta \in \mathbb S^1 $ we get the following statement: 
	\begin{enumerate}
		\item $\psi=0$: \[\operatorname{dim}  \left(  \mathcal L_{\leq 0} \right)_{\vartheta} = \sum_{j \in J} \Phi_{\widetilde P_j} +  \sum_{\{i \mid \psi_i \leq_{\vartheta} 0\}} \Phi_{P_i} \]
		\item $\psi \neq 0$, $\vartheta \in  \left(   \frac {\frac {\pi} {2}+ arg  \left( -\mu \left( 0\right)\right)} {q}, \frac {\frac {3\pi} {2}+ arg  \left( -\mu \left( 0\right)\right)} {q}\right) \: mod \: 			\frac {2\pi}{q}$: 
			\[\operatorname{dim}  \left( \mathcal L_{\leq \psi} \right)_{\vartheta}=  \sum_{j\in J} \Phi_{\widetilde P_j} +  \sum_{\{i \mid \psi_i \leq_{\vartheta} \psi\}} \Phi_{P_i} \]
		\item $\psi \neq 0$, $\vartheta \in  \left(  \frac {-\frac {\pi} {2}+ arg  \left( -\mu \left( 0\right)\right)} {q}, \frac {\frac {\pi} {2}+ arg  \left( -\mu \left( 0\right)\right)} {q}\right) \: mod \: \frac 			{2\pi}{q}$: 
			\[ \operatorname{dim}  \left( \mathcal L_{\leq \psi} \right)_{\vartheta}=  \sum_{\{i \mid \psi_i \leq_{\vartheta} \psi\}} \Phi_{P_i} \]
	\end{enumerate}
\end{corollary}
	
\begin{proof}
This is a direct consequence of the formal decomposition and the fact that for $\psi \neq 0$ we have
	\begin{align*}
		0 \leq_{\vartheta} \psi  &\Leftrightarrow  - \mu \left( t\right) t^{-q} \leq_{\vartheta} 0 \Leftrightarrow arg \left( -\mu \left( 0\right)\right) -q \vartheta \in \left(  \frac {\pi}{2}, \frac {3 \pi}		{2}\right) \\
		&\Leftrightarrow \vartheta \in  \left(  \frac {\frac {\pi} {2}+ arg  \left( -\mu \left( 0\right)\right)} {q}, \frac {\frac {3\pi} {2}+ arg  \left( -\mu \left( 0\right)\right)} {q}\right) \: mod \: \frac 		{2\pi}{q}
	\end{align*}
\end{proof}


\subsection{Topological description of the stalks}

As before we assume that $\mathcal M$ is a meromorphic connection with regular singularities along its divisor. 
Our aim is to determine the Stokes structure of $\mathcal N$  by using a topological point of view,  i.\:e. we will develop a topological description of the Stokes-filtered local system $(\mathcal L, \mathcal L_{\leq})$. Therefore we will use the following theorem:

\begin{theorem}[ \cite{Moc14}, Cor. 4.7.5]\label{theorem1} There is an isomorphism
$$\Omega: \mathcal L_{\leq \psi} := \mathcal H^0 \operatorname{DR}^{mod\: 0}_{ \widetilde{\Delta}}  \left( \mathcal N \otimes \mathcal E^{-\psi}\right) \rightarrow  \widetilde {\mathcal L}_{\leq \psi}:= \mathcal H^1 R \widetilde p_{\ast} \operatorname{DR}^{mod \: D}_{\widetilde X(D)}  \left( \mathcal M \otimes  \mathcal E^{\frac 1 y}\otimes \mathcal E^{-\psi}\right)$$
\end{theorem}
Here  $\widetilde X(D)$ denotes the real-oriented blow up of $X$ along the divisor component $D$, $\operatorname{DR}^{mod \: D}(\mathcal M)$ denotes the moderate de Rham complex of a meromorphic connection $\mathcal M$ on $X$ 
and $\widetilde p: \widetilde X(D) \to \widetilde{\Delta}$ corresponds to the projection $p$ in the real- oriented blow up space $\widetilde X(D)$ along $D$.\\

Theorem \ref{theorem1} is a special case of \cite{Moc14}, Cor. 4.7.5. In the following we will develop a topological view of the right hand side of the above isomorphism, which enables us to describe the Stokes-filtered local system $(\mathcal L, \mathcal L_{\leq})$ more explicitly.\\ 
First let us examine the behavior of ${\widetilde {\mathcal L}}_{\leq \psi}$ with respect to birational maps. 

\begin{proposition} Let $e: Z \to \Delta \times \mathbb P^1$ a birational map (i.e. a sequence of point blow-ups), $g \left( t,y\right) := \frac 1 y - \psi \left( t\right)$, $D_Z= e^{-1}  \left( D\right)$. Then: 
\[ \operatorname{DR}^{mod \: D}_{\widetilde {X} \left( D\right)} \left(  \mathcal M \otimes \mathcal E^{g \left( t,y\right)}\right) \cong R \tilde e_{\ast} \operatorname{DR}^{mod \: D_Z}_{\widetilde {Z}  \left( D_Z\right)} \left( e^+ \mathcal M \otimes \mathcal E^{g \left( t,y\right) \circ e}\right)\]
\end{proposition}

\begin{proof}
 We know that $e$ is proper. Furthermore we assumed that $\mathcal M$ is a meromorphic connection with regular singularities along $D$, i.e. particularly that $\mathcal M \otimes \mathcal E^{g \left( t,y\right)}$ is a holonomic $\mathcal D_X$-module and localized along $D$. Thus, using the fact that $e_+  \left( e^+ \mathcal M \otimes \mathcal E^{g \left( t,y\right) \circ e}\right) \cong  \mathcal M \otimes \mathcal E^{g \left( t,y\right)}$, we can apply \cite{Sab}, Prop. 8.9: 
	\begin{align*}
		\operatorname{DR}^{mod \: D} \left(  \mathcal M \otimes \mathcal E^{g \left( t,y\right)}\right)& \cong \operatorname{DR}^{mod \: D} \left( e_+ \left( e^+ \mathcal M \otimes 		\mathcal E^{g \left( t,y\right) \circ e}\right)\right) \\
                   & \cong R\tilde e_{\ast} \operatorname{DR}^{mod \: D_Z} \left( e^+ \mathcal M \otimes \mathcal E^{g \left( t,y\right) \circ e}\right)
	\end{align*}
\end{proof}
\begin{proposition}\label{prop1} Denote $D':= D \cup  \left( \Delta \times \{ \infty\}\right) $. Let $D'_Z:= e^{-1} \left( D'\right)$ and $\overline D_Z:= e^{-1}( \overline D)$. If $\overline D_Z$ is a normal crossing divisor, we have isomorphisms
\[R \tilde e_{\ast} \operatorname{DR}^{mod \: D_Z}_{\widetilde {Z}  \left( D_Z\right)} \left( e^+ \mathcal M \otimes \mathcal E^{g \left( t,y\right) \circ e}\right)\cong R \tilde e_{\ast} \operatorname{DR}^{mod \: D'_Z}_{\widetilde {Z}  \left( D'_Z\right)} \left( e^+ \mathcal M \otimes \mathcal E^{g \left( t,y\right) \circ e}\right) \]
and 
\[R \tilde e_{\ast} \operatorname{DR}^{mod \: D_Z}_{\widetilde {Z}  \left( D_Z\right)} \left( e^+ \mathcal M \otimes \mathcal E^{g \left( t,y\right) \circ e}\right)\cong R \tilde e_{\ast} \operatorname{DR}^{mod \: \overline D_Z}_{\widetilde {Z}  \left( \overline D_Z\right)} \left( e^+ \mathcal M \otimes \mathcal E^{g \left( t,y\right) \circ e}\right) \]
where $\tilde e$ denotes the induced map in the particular blow up spaces.
\end{proposition}
\begin{proof}
Assume $\overline D_Z$ a normal crossing divisor and consider the identity map $Z \xrightarrow{Id} Z$, which obviously induces an isomorphism on $Z \setminus \overline D_Z \to Z \setminus \overline D_Z $. We obtain `partial' blow up maps 
\[ \widetilde{Id}_1: \widetilde Z \left( D'_Z\right) \to \widetilde Z \left( D_Z\right) \text{ and } \widetilde{Id}_2: \widetilde Z \left( \overline D_Z\right) \to \widetilde Z \left( D_Z\right) \]
which induce the requested isomorphisms. These are variants of Prop.\:8.9 in \cite{Sab} (see also \cite{Sab}, Prop. 8.7 and Rem. 8.8). 
\end{proof}

\begin{corollary}
\[ \mathcal H^1  \left( R \widetilde p_{\ast} \operatorname{DR}^{mod \: D} \left(  \mathcal M \otimes \mathcal E^{g \left( t,y\right)}\right)\right) \cong  \mathcal H^1  \left( R \left(  \widetilde {p\circ e}\right)_{\ast} \operatorname{DR}^{mod  \left(  \star \right)}  \left( e^+ \mathcal M \otimes \mathcal E^{g \left( t,y\right) \circ e}\right)\right) \]
where $\star = D_Z$ (resp. $D'_Z$, resp. $\overline D_Z$).
\end{corollary}

As $(\mathcal L, \mathcal L_{\leq})$ was defined on $\partial \widetilde {\Delta}$ we will restrict our investigation to the boundaries of the relevant blow up spaces, i.\:e. we consider $\partial \widetilde{\Delta}$, $\partial \widetilde X  \left( D\right)$, $\partial \widetilde Z  \left( D_Z\right)$, $\partial \widetilde Z  \left( D'_Z\right)$ and $\partial \widetilde Z ( \overline D_Z)$. We will take a closer look at the stalks: 
\begin{lemma} \label{lem3} Let $\vartheta \in \mathbb  S^1$. There is an isomorphism
\footnotesize
	\begin{align*}
	 \left(  \mathcal H^1 \left( R  \left( \widetilde {p\circ e}\right)_{\ast} \operatorname{DR}^{mod  \left( D_Z\right)}  \left( e^+ \mathcal M \otimes \mathcal E^{g \left( t,y\right) \circ e}\right)	\right)\right)_{\vartheta} &\cong \\ 
	&\mathbb H^1 \left(  \left( \widetilde{p\circ e}\right)^{-1} \left( \vartheta\right), \operatorname{DR}^{mod  \left( D_Z\right)}  \left( e^+ \mathcal M \otimes \mathcal E^{g \left( t,y\right) 	\circ e}\right)\right)
	\end{align*}
\normalsize	
The same holds for $D'_Z$ and $\overline D_Z$ instead of $D_Z$.
\end{lemma}
	
\begin{proof}
$\widetilde {p\circ e}$ is proper, thus the claim follows by applying the proper base change theorem  (cf. \cite{Dim}, Th. 2.3.26, p. 41). 
\end{proof}
\\
In the following we will consider
$$ \mathcal F_\psi := \operatorname{DR}^{mod  \left( D'_Z\right)}  \left( e^+ \mathcal M \otimes \mathcal E^{g \left( t,y\right) \circ e}\right)$$
on the fiber $ \left( \widetilde{p\circ e}\right)^{-1} \left( \vartheta\right)$ in $\partial \widetilde Z  \left( D'_Z\right)$,  which in our case will be a  $1$-dimensional complex analytic space. Therefore we can assume $\mathcal F_\psi$ to be a perverse sheaf, i.e. a $2$-term complex $\mathcal F_\psi: 0 \to \mathcal F^0 \to \mathcal F^1 \to 0$ with additional conditions on the cohomology sheaves $\mathcal H^0  \left( \mathcal F_\psi\right)$ and $\mathcal H^1 \left( \mathcal F_\psi \right)$. Obviously $\mathcal H^i  \left( \mathcal F_\psi\right)=0$ for $i \neq 0,1$  ( cf. \cite{Dim}, Ex. 5.2.23, p. 139). 

Furthermore we can restrict $\mathcal F_{\psi}$ to the set $B_{\psi}^{\vartheta}:= \left\{ \zeta \in  \left( \widetilde{p\circ e}\right)^{-1} \left( \vartheta\right)  \mid  \left( \mathcal H^{\bullet}(\mathcal F_{\psi}) \right)_{\zeta} \neq 0\right\}$, which is an open subset of $ \left( \widetilde{p\circ e}\right)^{-1} \left( \vartheta\right)$. We denote the open embedding by 
$\beta_{\psi}^{\vartheta}: B_{\psi}^{\vartheta} \hookrightarrow  \left( \widetilde{p\circ e}\right)^{-1} \left( \vartheta\right)$.
Thus by interpreting  $\mathcal F_{\psi} $ as a complex of sheaves on $B_{\psi}^{\vartheta}$, we have  to  compute 
\[ \mathbb H^1 \left(  \left( \widetilde{p\circ e}\right)^{-1} \left( \vartheta\right), \beta_{\psi, !}^{\vartheta}  \mathcal F_{\psi} \right) \cong \mathbb H^1_c \left( B_{\psi}^{\vartheta}, \mathcal F_{\psi}\right).\]


\subsubsection{Construction of a Resolution of Singularities}

In the following sections we will construct a suitable blow up map $e$, such that we can describe $ \left( \widetilde {p\circ e}\right)^{-1} \left( \vartheta\right)$ and $B_{\psi}^{\vartheta}$ more concretely.
	
\begin{lemma} \label{lemma1}  Let $g \left( t,y\right)= \frac 1 y - \psi \left( t\right)$. There exists a sequence of blow up maps $e$ such that 
	\begin{enumerate}
		\item$g \circ e$ holomorphic or good, i.\:e. $(g \circ e)\left( u, v\right)= \frac 1 {u^m v^n} \beta \left( u,v\right)$, whereby $ \beta$ holomorphic and $ \beta \left( 0,v\right)				\neq0$. 
		\item For all $i$ the strict transform of $S_i$ intersects  $D_Z$  in a unique point $P_i$.
	\end{enumerate}
\end{lemma}

\begin{proof}
The divisor components $S_i$ are given by $S_i: \mu_i \left( t\right) y= t^{q_i}$. Let $n:= max \{ q_i\}$. We distinguish two cases: 
	\begin{enumerate}
	\item $\psi=0$, i. e.  $g \left( t,y\right)= \frac 1y$. After $n$ blow ups in $ \left( 0,0\right)$ we get an exceptional divisor $D_Z$ with local coordinates
		\[ t= u_k v_k, y= u^{k-1}_k v^k_k \text{ and } t= u_{\tilde n}, y=\tilde u^n_{n} \tilde v_{n}\]
		and $g \circ e$ is good or holomorphic in every point. 
			\begin{center}
				\includegraphics[scale=0.8]{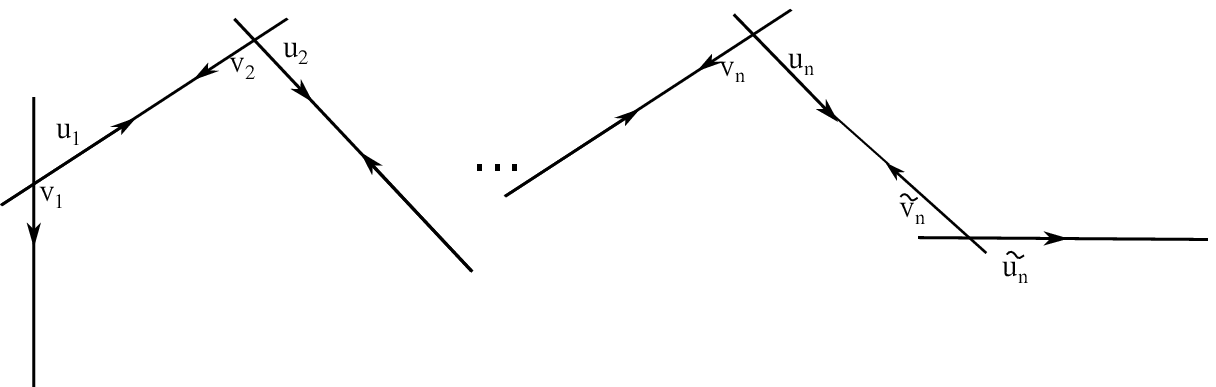}
			\end{center}
		The intersection points $P_i$ of the strict transform of $S_i: \mu_i \left( t\right) y= t^{q_i}$ is given by 
		\[P_i=  \left( 0, \frac 1 {\mu_i \left( 0\right)}\right)\] 
		in the proper coordinates (i.\:e.  $ \left( u_{q_i +1}, v_{q_i+1}\right)$ for $q_i < n$ and $\left( \tilde u_{n}, \tilde v_{n} \right)$ for $q_i=n$)
	\item $\psi \neq 0$: $\psi$ is given by $\psi \left( t\right) = \mu \left( t\right) t^{-q}$,  so $g \left( t,y\right) = \frac {t^q-\mu \left( t\right) y} {y t^q}$.$ \left( g\circ e\right)$ is good in every 		point except: 
		\begin{itemize}
			\item $q<n, k= q+1:$ $P= \left( 0, \frac 1 {\mu \left( 0\right)}\right)$ with local coordinates $ \left( u_{q+1}, v_{q+1}\right)$
			\item $q=n, k=n:$ $P= \left( 0, \frac 1 {\mu \left( 0\right)}\right)$ with local coordinates $ \left( \tilde u_{n}, \tilde v_{n}\right)$
		\end{itemize}
		Let $q<n$. After  changing coordinates $u' = u_k, v'= v_k -\frac 1 {\mu \left( u_k v_k\right)}$ and after $q$ blow-ups in $ \left( 0,0\right)$,  $ \left( g \circ e\right) $ is good for 		every point of $D_Z$ in local coordinates $ \left( u'_s, v'_s\right)$ and holomorphic for every point of $D_Z$ in local coordinates $ \left( \tilde u'_{q}, \tilde v'_{q}\right)$	
			\begin{center}
				\includegraphics[scale=0.8]{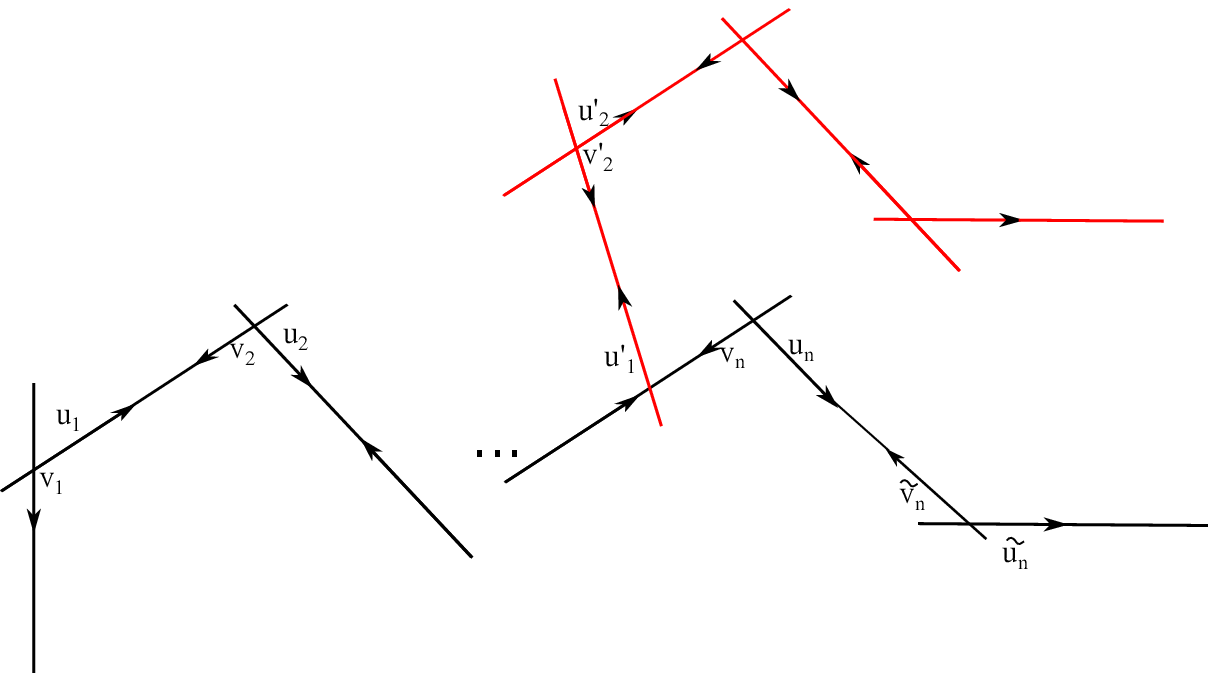}
			\end{center}
		Let $q=n$: After a change of coordinates $u'= \tilde u_{n}, v' = \tilde v_{n} - \frac 1 {\mu \left( \tilde u_{n}\right)}$ and $n$ blow-ups in $ \left( 0,0\right)$, $ \left( g \circ e\right) 		$ is good for every point of $D_Z$ in local coordinates $ \left( u'_s, v'_s\right)$ and holomorphic for every point of $D_Z$ in local coordinates $ \left( \tilde u'_{n}, \tilde v'_{n}		\right)$.
		As before the intersection points $P_i$ of the strict transform of $S_i: \mu_i \left( t\right) y= t^{q_i}$ wit $D_Z$ is given by 
		\[P_i=  \left( 0, \frac 1 {\mu_i \left( 0\right)}\right)\] 
		in the suitable coordinates.
		Now let $q_i=q$ and $\mu_i \left( 0\right) = \mu \left( 0\right)$, i.\:e. we consider $S: \mu \left( t\right) y= t^{q}$ (Notice, that $S$ corresponds to our given $\psi$!)
		\begin{itemize}
			\item $q<n, \text{ i.e. } k:=q+1\neq n$: $\overline S: u^{k-1}_{k} v^{k-1}_{k}  \left( 1- \mu_i \left( u_{k}v_{k}\right) v_{k}\right)=0$ and by coordinate transform we have: $ 				\overline S: u^{k-1}_{k} v^{k-1}_{k} v'=0$. As $v'= \tilde u'_{q} \tilde v'_{q}$ we get the unique intersection point $P= \left( 0, 0\right)$.
			\item $q=n$: In the same way we get the intersection point $P= \left( 0, 0\right)$.
		\end{itemize}
	\end{enumerate}
\end{proof}


\subsubsection {Topology of   $ \left( \widetilde {p \circ e}\right)^{-1}  \left( \vartheta\right)$}

$D'_Z$ is a normal crossing divisor, so locally at a crossing point $D'_Z$ has the form $ \{ u v=0\}$ and at a smooth point $D'_Z$ has the form $\{ u=0\}$. Thus we can describe $\partial \widetilde Z  \left( D'_Z\right)$ in local coordinates: 
	\begin{itemize}
		\item real blow up with respect to $\{ u=0\}$: $\zeta =  \left( 0, \theta_{u}, \left| v \right|, \theta_{v}\right)$ (with $v= \left| v \right| \cdot e^{i \theta_{v}}$)
		\item real blow up with respect to $\{ u v=0\}$: $\zeta =  \left( 0, \theta_{u}, 0, \theta_{v}\right)$
	\end{itemize}
	Now we take the fiber of $\vartheta \in \mathbb  S^1$, i.\:e. we consider $ \left( \widetilde {p \circ e}\right)^{-1}  \left( \vartheta\right)$. For every fixed $\left | v \right|$ we have a bijection $\{ \left(  0, \theta_u,\left | v \right|, \theta_v \right)\} \leftrightarrow \mathbb S^1$, thus, following the nomenclature of C. Sabbah, we can interpret $ \left( \widetilde {p \circ e}\right)^{-1}  \left( \vartheta\right)$ as a system of pipes, which furthermore is homeomorphic to a disc (\cite{Sab}, p. 203 )

Note, that for all $i$ the strict transform of the divisor component $S_i$ intersects the real blow up divisor $ \left( \widetilde {p \circ e}\right)^{-1} \left( \vartheta\right)$ in a unique point $P_i$.

\begin{remark} 
	\begin{enumerate}
		\item Also the irreducible components $\widetilde S_j$ intersect $ \left( \widetilde {p \circ e}\right)^{-1} \left( \vartheta\right)$ in distinct points $\widetilde P_j$. This follows 			directly from Assumption \ref {Ass 2}.
		\item According to Assumption \ref {Ass 1} we have ($q_i \neq q_j$ or $\mu_i \left( 0\right) \neq \mu_j \left( 0\right)$ for $i \neq j$). This induces $P_i \neq P_j$.  
		\item Every intersection point $P_i, \widetilde P_j$ may be interpreted as a 'leak' in the system of pipes $ \left( \widetilde {p \circ e}\right)^{-1} \left( \vartheta\right)$ (see 				\cite{Sab}, p. 204). Thus topologically we can think of $ \left( \widetilde {p \circ e}\right)^{-1} \left( \vartheta\right)$ as a disc with singularities, which come from the 				intersection with $\widetilde S_j$ and $\overline{S_i}$. 
	\end{enumerate}
	$$\begin{xy}
		\xymatrixrowsep{1.5cm}
		\xymatrixcolsep{2.0cm}
 		 \xymatrix
  		{
		\psi =0 & \psi \neq 0\\
		 \includegraphics[scale=0.3]{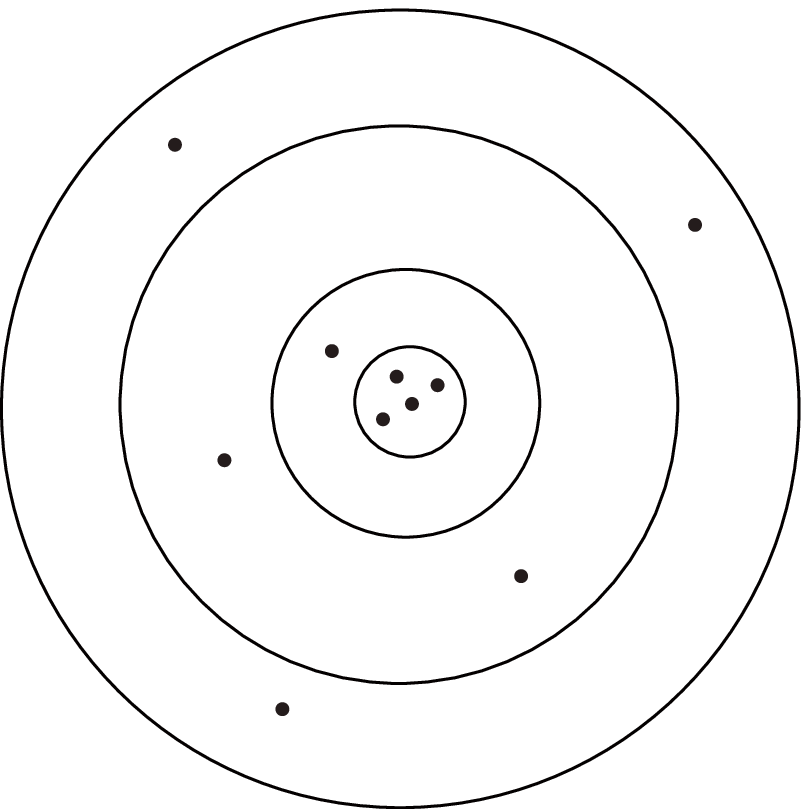} &  \includegraphics[scale=0.3]{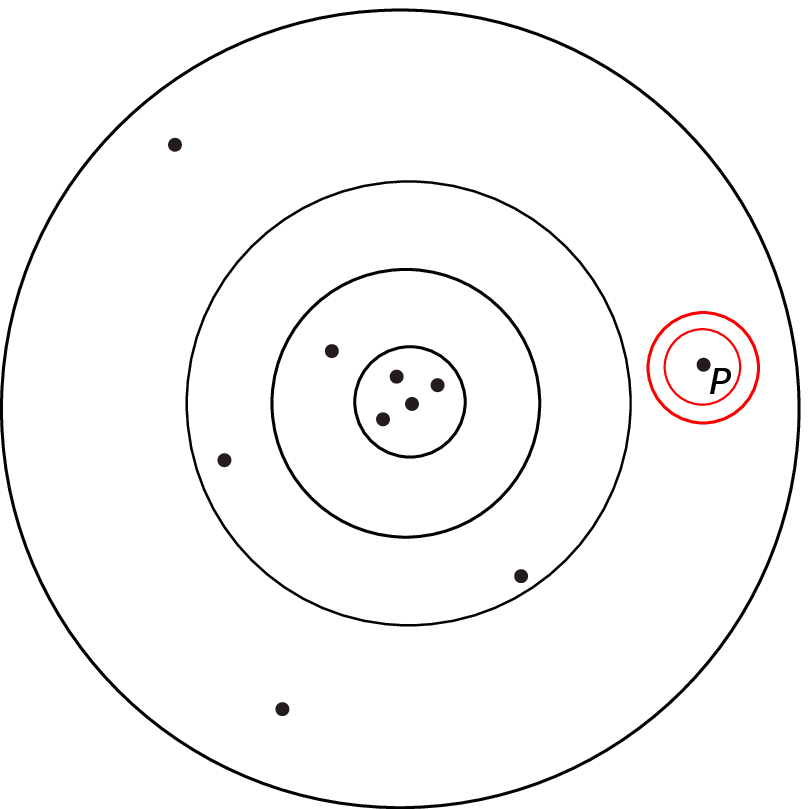}
   		}
		\end{xy}$$
\end{remark}	

\subsubsection{Explicit Description of  $B_{\psi}^{\vartheta}$}

Remember the definition of $B_{\psi}^{\vartheta} := \left\{\zeta \in  \left( \widetilde {p\circ e}\right)^{-1} \left( \vartheta\right)  \mid  \left( \mathcal H^{\bullet} \mathcal F_{\psi} \right)_{\zeta} \neq 0\right\}$.
Obviously we have: 
\[ \zeta \in B_{\psi}^{\vartheta} \Leftrightarrow  \left( \mathcal H^{\bullet}\mathcal F_{\psi} \right)_{\zeta} \neq 0  \Leftrightarrow e^{ \left( g \circ e\right)} \in \mathcal A^{mod \: D'_Z}_{\vartheta} \text{ near } \zeta\]
(The second equivalence follows by considering $\operatorname{DR}^{mod\: \overline D_Z}  \left(  e^+ \mathcal M \otimes \mathcal E^{g \left( t,y\right) \circ e}\right)$, whereof we know that it has cohomology in degree $0$ at most (cf. Proposition \ref{prop1} and Lemma \ref {DRlemma} ).) Thus we need to take a closer look at the exponent $g \circ e$. Therefore we use the following Lemma (cf. \cite{Sab}, 9.4):

\begin{lemma} \label{lem2} Let $u, v$ local coordinates of the divisor $D_Z$,  such that $f \left( u, v\right)$ holomorphic or good, i.\:e. 
$$f \left( u, v\right)= \frac 1 {u^m v^n} \beta \left( u,v\right), \: \text{ whereby } \beta \text{ holomorphic and } \beta \left( 0,v\right)\neq0. $$
Then $e^{f \left( u,v\right)} \in \mathcal A_{\vartheta}^{mod \: D_Z}$ around a given point  $\zeta \in   \left( \widetilde{p\circ e}\right)^{-1}  \left( \vartheta\right)$ if and only if
	\begin{align*}
		f \text{ holomorphic in }\zeta \text { or } arg  \left( \beta \left( 0,v\right)\right)-m \theta_{u}-n \theta_{v} \in  \left( \frac {\pi} 2, \frac{3\pi} 2\right) \: mod \: 2 \pi .
	\end{align*}
\end{lemma}	

\begin{lemma} \label{lem1} 
	\begin{enumerate}
		\item Let $S_i: \mu_i \left( t\right) y= t^{q_i}$, $\psi_i \left( t\right)= \mu_i \left( t\right) t^{-q_i}$  with $\mu_i \left( 0\right)\neq0$ and $P_i$ the corresponding intersection point. 			Then we have: $ P_i \in B_{\psi}^{\vartheta} \Leftrightarrow \psi_i \leq_{\vartheta} \psi$.
		\item Let $\widetilde P_j$ the intersection points of $\widetilde S_j$ with $ \left( \widetilde {p \circ e}\right)^{-1} \left( \vartheta\right)$. Then we have: $ \left\{\widetilde P_j 			\mid j = 1, \ldots, J \right\} \subset B_{\psi}^{\vartheta} \Leftrightarrow 0 \leq_{\vartheta} \psi $.
	\end{enumerate}
\end{lemma}

\begin{proof} 
This follows by examining goodness condition of $g \circ e$ near the intersection points and by Remark  \ref{remark1}
\end{proof}

\begin{definition} For abbreviation we define the following sets:
	\begin{itemize}
		\item $\mathcal P^{\vartheta}: = \{\widetilde P_j \mid j = 1 \ldots, J \} \cup \{P_i \mid i = 1, \ldots I \} \subset  \left( p\circ e\right)^{-1}  \left( \vartheta\right)$
		\item $\mathcal P^{\vartheta}_{\psi} := \mathcal P^{\vartheta} \cap B_{\psi}^{\vartheta}$
	\end{itemize} 
\end{definition}

\begin{lemma} \label{I_0} Let $\psi=0$. Then the fundamental group $\pi_1 \left( B^{\vartheta}_0 \setminus \mathcal P^{\vartheta}_{0} \right)$  is a free group of rank $\#  \left( \mathcal P^{\vartheta}_{0}\right)$.
\end{lemma}
		
\begin{proof}
This follows from the fact, that $B^{\vartheta}_0$ emerges from glueing the following sets of points: 
	\begin{itemize}
		\item $ M_1 = \{ \zeta = \left( 0, \theta_{u_1}, \left |v_1\right|, \theta_{v_1}\right)  \mid v_1 \neq 0\}$ 
		\item $ M_2 = \{ \zeta = \left( 0, \theta_{u_1}, 0, \theta_{v_1}\right) \mid \theta_{v_1}  \in \left(   \frac {\pi} {2}, \frac {3 \pi } {2} \right)\}$
		\item $M_3 = \{ \zeta = \left( 0, \theta_{u_k}, \left |v_k\right|, \theta_{v_k}\right) \mid \theta_{v_k}  \in \left(   \frac {\pi} {2}, \frac {3 \pi } {2} \right) \}$
		\item $M_4 = \{ \zeta = \left( 0, \theta_{\tilde u_{n}}, \left |\tilde v_{n}\right|, \theta_{\tilde v_{n}}\right) \mid \theta_{\tilde v_{n }}  \in \left(   \frac {\pi} {2} - n \vartheta, \frac {3 \pi } 		{2} - n \vartheta \right)\}$
	\end{itemize}
Since $M_1$ contains the $\widetilde P_j$'s and $M_2, M_3, M_4$ are simply connected and contain the relevant $P_i$'s, this shows the claim. 
\end{proof}

\begin{lemma} \label{I_1} Let $\psi \neq 0$, i.\:e. $\psi$ is given by $\psi \left( t\right)= \mu \left( t\right) t^{-q}$. Then $\pi_1 \left( B_{\psi}^{\vartheta} \setminus \mathcal P^{\vartheta}_{\psi}\right)$ is a free group of rank  $ \#  \left( \mathcal P^{\vartheta}_{\psi}\right)$.
\end{lemma}	

\begin{proof}
Explicitely we have to prove:  
	\begin{enumerate}
		\item For $\vartheta \in \left(  \frac {\frac {\pi} {2}+ arg  \left( -\mu \left( 0\right)\right)} {q}, \frac {\frac {3\pi} {2}+ arg  \left( -\mu \left( 0\right)\right)} {q} \right) \: mod \: \frac {2\pi}{q}			$:  $\pi_1 \left( B_{\psi}^{\vartheta} \setminus \mathcal P^{\vartheta}_{\psi}\right)$ is a free group of rank $\# \{\widetilde P_j \} + \#\{P_i\mid P_i \in B_{\psi}^{\vartheta}\}$
		\item For $\vartheta \notin \left(  \frac {\frac {\pi} {2}+ arg  \left( -\mu \left( 0\right)\right)} {q}, \frac {\frac {3\pi} {2}+ arg  \left( -\mu \left( 0\right)\right)} {q} \right) \: mod \: \frac {2\pi}			{q}$: $ \pi_1 \left( B_{\psi}^{\vartheta} \setminus \mathcal P^{\vartheta}_{\psi}\right)$ is a free group of rank \linebreak $\#\{P_i\mid P_i \in B_{\psi}^{\vartheta}\}$
	\end{enumerate}
For $k \leq q$ we have: $\zeta=  \left( 0, \theta_{u_k}, \left |v_k\right|, \theta_{v_k}\right) \in B_{\psi}^{\vartheta} \Leftrightarrow  arg  \left( -\mu \left( 0\right)\right) - q \vartheta \in \left(   \frac {\pi} {2}, \frac {3 \pi } {2} \right)$. This explains the sub-division into the two cases above. As in the previous lemma $B_{\psi}^{\vartheta}$ consists of glueing simply connected sets (additionally taking the sets of the branching part into account).	
\end{proof}
\\
Descriptively, this means, that there are no other holes' in $B_{\psi}^{\vartheta}$ than the singularities $P_i$, which arise  from the intersection with $S_i$ and possibly -- 	depending on the choice of $\vartheta$ if $\psi \neq 0$ -- the singularities $\widetilde P_j$, which arise from the intersection with $\widetilde S_j$. 

Remark, that an open interval of the boundary of $B_{\psi}^{\vartheta}$ lies in $B_{\psi}^{\vartheta}$! This holds because for $k=n$ we have
\[ \left(  0, \vartheta, 0, \theta_{\tilde v_{n}}  \right)  \in B_{\psi}^{\vartheta} \Leftrightarrow  \theta_{\tilde v_{n}} \in  \left(   \frac {\pi} {2} - n \vartheta, \frac {3 \pi } {2} -n \vartheta \right)\]
both if $q<n$ and  if $q=n$.


\subsubsection {Dimension of $\mathbb H^1_c \left( B_{\psi}^{\vartheta}, \mathcal F_{\psi} \right)$}

\begin{proposition} On $B_{\psi}^{\vartheta} \setminus \mathcal P^{\vartheta}_{\psi}$ the perverse sheaf $\mathcal F_{\psi} $ has cohomology in degree $0$ at most. 
\end{proposition}

\begin{proof}
We know that $e^+ \mathcal M $ has regular singularities along the divisor and that $\mathcal E^{g \left( t,y\right)}$ is good or even holomorphic. Moreover the divisor is normal crossing except possibly at the intersection points $P_i$, $\widetilde P_j$. Thus we can apply \cite{Sab}, Corollary 11.22.
\end{proof}

\begin{lemma}[\cite{SabND}, Prop 1.1.6] \label{Delta} Let $\Delta$ be an open disc and $\Delta'$ an open subset of the closure $\overline{\Delta}$, consisting of $\Delta$ and a connected open subset of $\partial \overline{\Delta}$. Let $C$ be a finite set of points in $\Delta$. Let $\mathcal M$ be a regular singular $\mathcal D_{\Delta'}$-module and consider a perverse sheaf $\mathcal F$ with singularities in $c \in C$ only. Then $\mathbb H^k_c \left( \Delta', \mathcal F \right)=0$ for $k \neq 1$ and  $\operatorname{dim} \: \mathbb H^1_c \left( \Delta', \mathcal F \right)$ is equal to the sum of the dimensions of the vanishing cycle spaces at $c \in C$.
\end{lemma}

\begin{proof}
For a proof we refer to \cite{SabND}.
\end{proof}

As a direct consequence we see:
\begin{theorem}\label{surj} For $\vartheta \in \mathbb S^1$ we have $\operatorname{dim}  \left( \mathcal L_{\leq \psi} \right)_{\vartheta}= \operatorname{dim}  \left( \widetilde{\mathcal L}_{\leq \psi} \right)_{\vartheta}$.
\end{theorem}
		
\begin{proof} This follows from Corollary \ref {cor1} and Lemma \ref{Delta}. 
\end{proof}

\begin{remark} \label{remarknew}
In an analogous way we can determine $\mathcal L$ on $\mathbb S_1$. First let us recall that $\mathcal L= j_{\ast}  \mathcal L'$, whereby $\mathcal L'$ denotes the local system associated to $\mathcal N:= \mathcal H^0 p_+\left(\mathcal M \otimes \mathcal E^{\frac 1 y}\right)$. Consider a small circle $\mathbb S^1_{\epsilon}$ around $0 \in \Delta$. Then we have
 \[ \mathcal L_{\mid S^1} \cong \mathcal L'_{\mid S^1_{\epsilon}}\]
 On $\mathbb S^1_{\epsilon}$ we can identify $\mathcal L' = \mathcal H^0 \operatorname{DR}^{mod \: 0}_{\widetilde {\Delta}} \left(\mathcal N\right)$ (since the growing condition $mod \: 0$ is irrelevant outside of $\partial \widetilde {\Delta})$. Furthermore in this situation, the isomorphism 
 \[\Omega:  {\mathcal H^0 \operatorname{DR}_{\widetilde \Delta}^{mod \: 0}\left(\mathcal N\right)}_{\mid \mathbb S^1_{\epsilon}} \xrightarrow{\cong} {\mathcal H^1 R \widetilde{p}_{\ast} \operatorname{DR}^{mod \: D} \left(\mathcal M \otimes \mathcal E^{\frac 1 y}\right)}_{\mid S^1_{\epsilon}}\]
 holds (analogously to Theorem \ref{theorem1}). 
Now as before we can describe the right hand side in topological terms. Since $\mathcal E^{\frac 1 y}$ is a good (or even holomorphic) connection, we know that $\operatorname{DR}^{mod \: D} \left(\mathcal M \otimes \mathcal E^{\frac 1 y}\right)$ has cohomology in degree zero at most and thus corresponds to a sheaf. By taking the stalk at a point $\rho \in \mathbb S^1_{\epsilon}$ we get 
\[{\left(\mathcal H^0 \operatorname{DR}_{\widetilde \Delta}^{mod \: 0}\left(\mathcal N\right)\right)}_{\rho} \cong H^1_c\left(\widetilde p^{-1}(\rho), \operatorname{DR}^{mod \: D} \left(\mathcal M \otimes \mathcal E^{\frac 1 y}\right)\right). \]
$\widetilde p^{-1}(\rho)$ corresponds to the projective line with a real blow up in the point $\infty$ (denoted by $\mathbb S^1_{\infty}$). We can interpret this space as a disc with boundary $\mathbb S^1_{\infty}$. Furthermore by choosing $\epsilon$ small enough, we know that every irreducible component $S_i$ resp. $\widetilde S_j$ of $SS(\mathcal M)$ meets the (inner of the) disc in exactly one point ($P_i$ resp. $\widetilde P_j$).  \\
The sheaf is supported everywhere away from the boundary, on the boundary its supported on an open hemicycle of $\mathbb S^1_{\infty}$ (i.e. the points where it fulfills the moderate growth condition).  Thus the dimension of $H^1_c\left(\widetilde p^{-1}(\rho), \operatorname{DR}^{mod \: D} \left(\mathcal M \otimes \mathcal E^{\frac 1 y}\right)\right)$ depends on the dimension of the vanishing cycle spaces in the points $P_i$ and $\widetilde P_j$:
\[ \operatorname{dim}\:  H^1_c\left(\widetilde p^{-1}(\rho), \operatorname{DR}^{mod \: D} \left(\mathcal M \otimes \mathcal E^{\frac 1 y}\right)\right) = \operatorname{dim} \sum_{i \in I} \Phi_{P_i} + \sum _{j \in J} \Phi_{\widetilde P_j}.\] 
\end{remark}
We will use this topological description in the following chapter,  supply it to the explicit example an determine a set of linear Stokes data.

\vspace*{1cm}

\section {Proof of Theorem \ref{theorem2}: Explicit example for the determination of Stokes data }
As already mentioned in the introduction, we will consider the following explicit example: \\
Let $X= \Delta \times \mathbb P^1$ and $\mathcal M$ a meromorphic connection on $X$ of rank $r$ with regular singularities  along its divisor. Let the singular locus $SS \left( \mathcal M\right)$ be of the following form:  $SS \left( \mathcal M\right)= \{ t \cdot y \cdot  \left( t-y\right) \cdot x=0 \}$. Denote the irreducible components  ($S_1: y=t$) and ($\widetilde S_1: x=0$).

	\begin{center}
		\begin{tikzpicture}[domain=-1.3:1.3] 
		\draw[-] ( 0,-4.2) -- ( 0,0.2) node[above] {};
		\draw [color=red] [-] ( -2,-4) -- ( 2,-4) node[right] {$\Delta \times \{0\}= \widetilde S_1$};
		\draw[-] ( -2,0) -- ( 2,0) node[right]{$\Delta \times \{ \infty\}$};
  		\draw[color=blue] plot ( \x,{ -\x }) node[right] {$S_1$};
		\end{tikzpicture}
	\end{center}
The following chapter is devoted to prove Theorem \ref{theorem2} and thereby  determine a set of Stokes data to $\mathcal H^0 p_+ \left( \mathcal M \otimes \mathcal E^{\frac 1 y}\right)$. 


\subsection{Stokes-filtered local system}

According to Theorem \ref{form_dec},   $\hat {\mathcal N} = \mathcal H^0 p_+ \left( \mathcal M \otimes \mathcal E^{q}\right)^{\wedge}_0$ can be decomposed to: 
\[  \hat {\mathcal N}\cong R_0 \oplus  \left( R_1 \otimes \mathcal E^{\frac 1 t}\right)\]
whereby $rk \left( R_0\right)= rk \left( R_1\right)=r$. Now we will describe the Stokes structure via the isomorphism  $\Omega$ (Theorem \ref{theorem1}). $\Omega$ identifies $ \left( \mathcal L_{\leq\psi}\right)_{\vartheta}  \cong H^1 \left(  \left( \widetilde{p \circ e}\right)^{-1} \left( \vartheta\right), \beta^{\vartheta}_{\psi, !} \mathcal F_{\psi}\right)$.

\begin{lemma} \label{BU} There exists a sequence of point blow ups $e: Z \to X$, such that 
	\begin{enumerate}
		\item the singular support of $\mathcal M$ becomes a normal crossing divisor
		\item for both $\psi=0$ and $\psi= \frac 1 t$ the exponent $g\circ e$ is holomorphic or good in every point. 
	\end{enumerate}
\end{lemma}

\begin{proof} 
The following blow up of the divisor satisfies the requested conditions:
\vspace*{0.3cm}
\begin{center}
\includegraphics[scale=0.8]{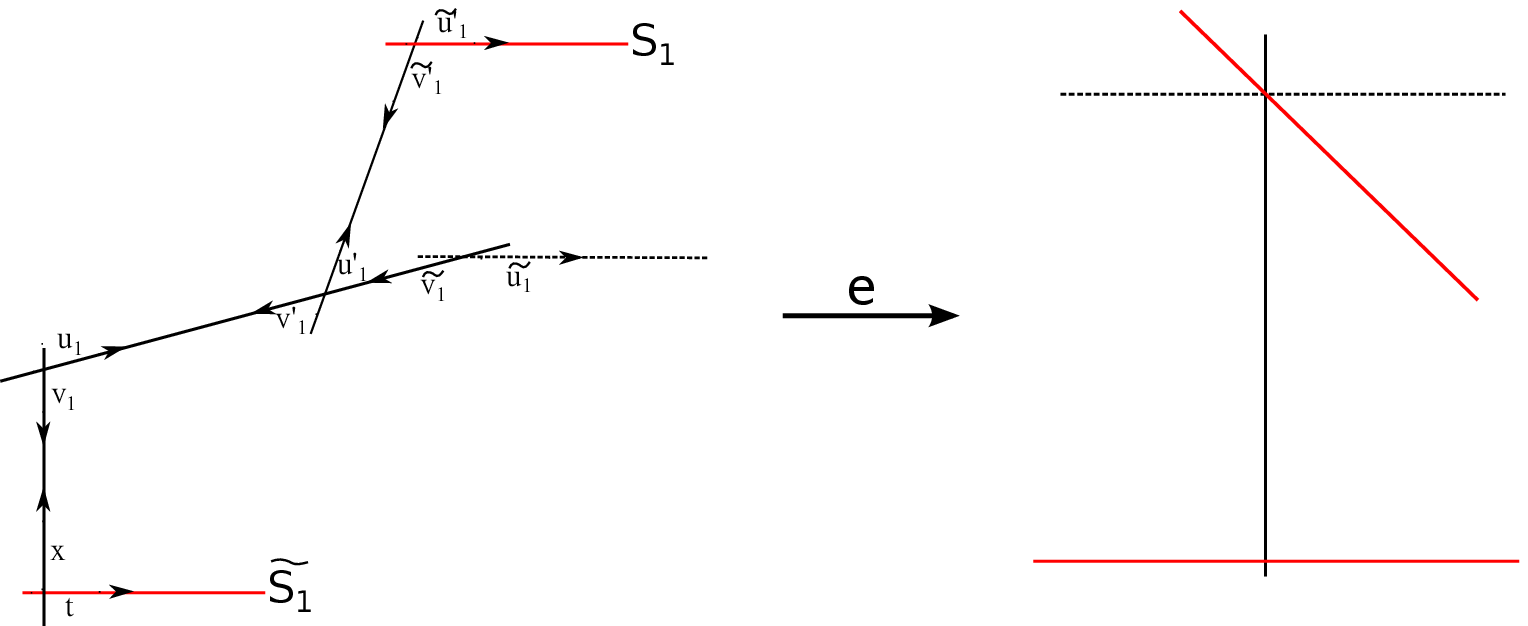}
\end{center}
(i.\:e. a point-blow-up in $(t,y)=0$ and a second point-blow-up in $(\tilde u_1, \tilde v_1)=(0,1)$)
\end{proof}

We will denote the resulting normal crossing divisor by $\overline D_Z$. The fiber over $\vartheta$ with respect to the blow up along $\overline D_Z$ is  homeomorphic to a closed disc with two `holes'. We will denote it by $\overline A \times \{ \vartheta \}$.

\begin{lemma}\label{DRlemma} $\operatorname{DR}^{mod \: \overline D_Z} \left( e^+ \mathcal M \otimes \mathcal E^{g\circ e}\right)$ has cohomology in degree $0$ at most. 
\end{lemma}

\begin{proof}
Since $\overline D_Z$ is normal crossing and $g \circ e $ is good or holomorphic  (Lemma \ref{BU}) the claim follows by  \cite{Sab}, Prop. 8.15 and Cor. 11.22.
\end{proof}
\\
Consider the map $\kappa: \widetilde Z \left( \overline D_Z\right) \to \widetilde Z \left( D_Z\right)$. Restricting it to a fiber $\kappa_{\vartheta}: \overline A \times \{ \vartheta\}  \to  \left( \widetilde{p \circ e}\right)^{-1} \left( \vartheta\right)$, it is just the identity except at the points $\widetilde P_1$ and $P_1$, where it describes the "collapse" of the real blow ups of  $\widetilde P_1$ and $P_1$ back to these points.  $\operatorname{DR}^{mod \: \overline D_Z} \left( e^+ \mathcal M \otimes \mathcal E^{g\circ e}\right)$ has cohomology in degree $0$ at most and therefore corresponds to a local system $\mathcal G_{\psi}$.  Then $\operatorname{DR}^{mod \: D_Z} \left( e^+ \mathcal M \otimes \mathcal E^{g\circ e}\right)$ corresponds to the perverse sheaf $\mathcal F_{\psi}$ given by $0 \to \kappa_{\ast} \mathcal G_{\psi} \to 0 \to 0$. Let $\mathcal G_{\psi}^{\vartheta}$ the restriction to the support $\kappa^{-1} \left(  B_{\psi}^{\vartheta}\right)$ in the fiber $\overline A \times \{ \vartheta\}$ and  $\overline{\beta}^{\vartheta}_{\psi}: \kappa^{-1} \left( B_{\psi}^{\vartheta}\right) \hookrightarrow \overline A \times \{\vartheta\}$ the open inclusion. Because of Proposition \ref{prop1} we have:
\[ H^1 \left( \overline A \times \{\vartheta\}, \overline{\beta}^{\vartheta}_{\psi, !} \mathcal G_{\psi}^{\vartheta} \right) \cong  H^1 \left(   \left( \widetilde{p \circ e}\right)^{-1} \left( \vartheta\right), \beta^{\vartheta}_{\psi, !} \mathcal F_{\psi}^{\vartheta} \right) \]
Furthermore we denote by $\mathcal K$ the local system  on $\widetilde Z \left( \overline D_Z\right)$ corresponding to the pullback connection $e^+ \mathcal M$ of rank $rk  \left( \mathcal M\right)$ (see Cor. 8.3 in \cite{Sab}) and by $\mathcal K^{\vartheta}$ its restriction to  $\overline A \times \{\vartheta\}$. Then obviously $ \overline{\beta}^{\vartheta}_{\psi, !} \mathcal G^{\vartheta}_{\psi} $ equals $\overline{\beta}^{\vartheta}_{\psi, !} {\overline{\beta}^{\vartheta}_{\psi}}^{-1} \mathcal K^{\vartheta}$ for all $\psi$. For brevity we will write  $\overline{\beta}^{\vartheta}_{\psi, !} \mathcal K^{\vartheta}$ instead of  $\overline{\beta}^{\vartheta}_{\psi, !} {\overline{\beta}^{\vartheta}_{\psi}}^{-1} \mathcal K^{\vartheta}$ in the following.
 Thus it is enough to examine $ \left( \mathcal L_{\leq\psi}\right)_{\vartheta}  \cong H^1 \left(  \overline A \times \{\vartheta\}, \overline{\beta}^{\vartheta}_{\psi, !} \mathcal K^{\vartheta}\right)$. With the blow up $e$ constructed in the proof of  Lemma \ref{BU} we can determine the open subset $B^{\psi}_{\vartheta} \subset  \left( \widetilde{ p \circ e}\right)^{-1} \left( \vartheta\right)$ and
we receive the following pictures, which show $ \left( \widetilde{p\circ e}\right)^{-1} \left( \vartheta\right)$ (resp. $\overline A \times \{ \vartheta \}$) and the subsets $B_0^{\vartheta}$, $B_{\frac 1 t}^{\vartheta}$. One can see very clearly that by passing the Stokes directions $±\frac {\pi} 2$, the relation of the subsets  $B_0^{\vartheta}$ and $B_{\frac 1 t}^{\vartheta}$ changes from $B_0^{\vartheta} \subset B_{\frac 1 t}^{\vartheta}$ to $B_{\frac 1 t}^{\vartheta} \subset B_0^{\vartheta}$ and vice versa.

	$$\begin{xy}
		\xymatrixrowsep{0.85cm}
		\xymatrixcolsep{0.2cm}
 		\xymatrix
		{
		\vartheta = 0 & \vartheta = \frac {\pi} 2- \epsilon & \vartheta = \frac {\pi} 2 +\epsilon & \vartheta = \pi \\
		\includegraphics[scale=0.3]{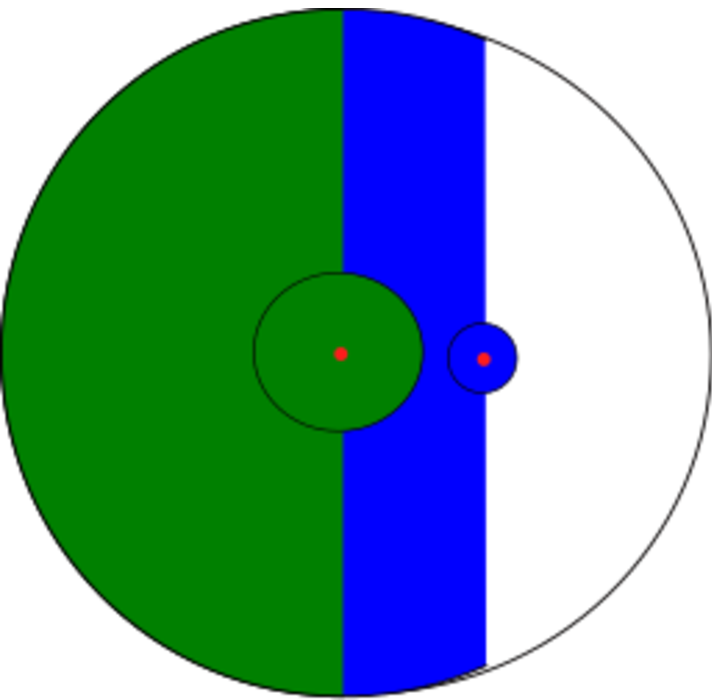} & \includegraphics[scale=0.3]{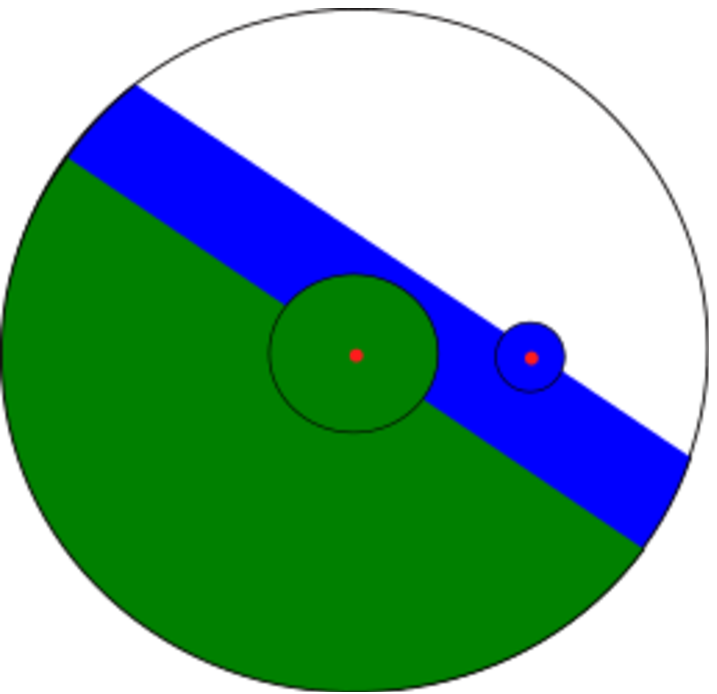}& \includegraphics[scale=0.3]{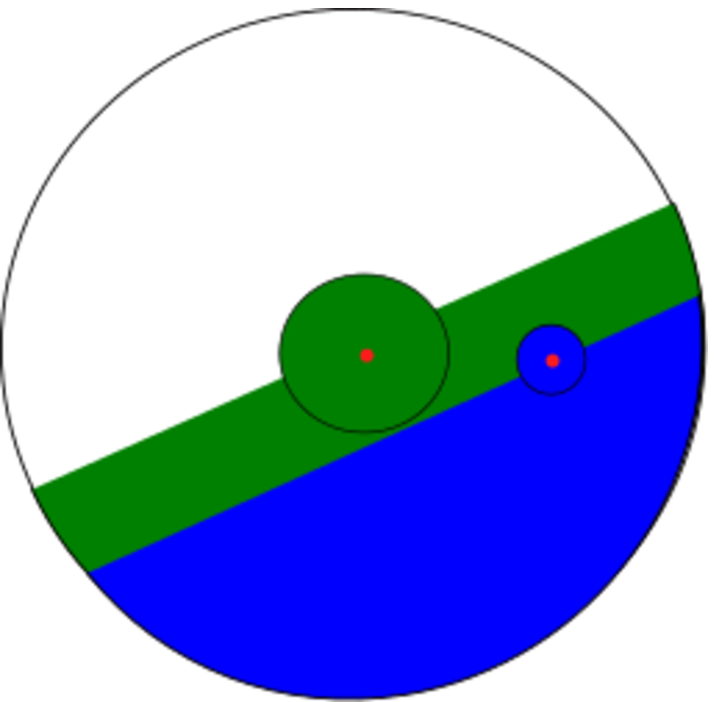} &\includegraphics[scale=0.3]{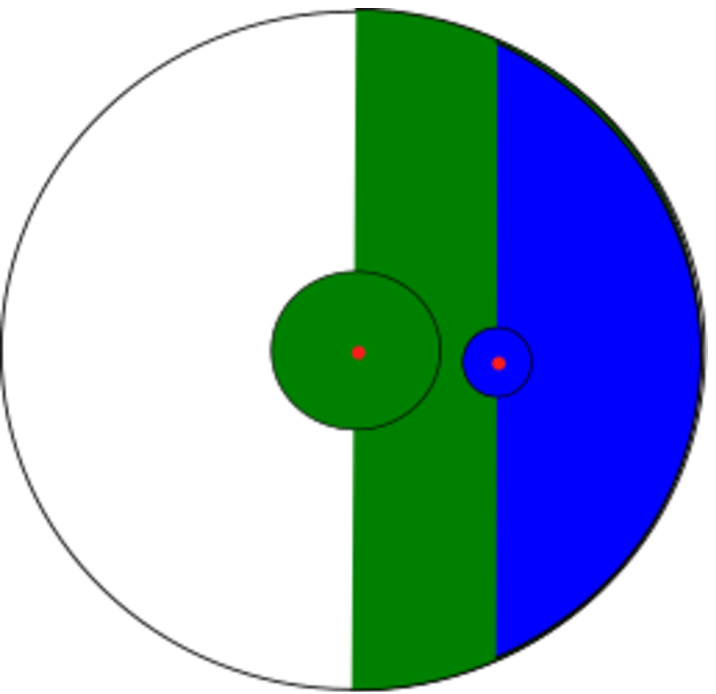}
 		 }
	\end{xy}$$

	$$\begin{xy}
		\xymatrixrowsep{0.85cm}
		\xymatrixcolsep{0.2cm}
 		\xymatrix
		{
		\vartheta = \pi & \vartheta = \frac {3\pi} 2- \epsilon & \vartheta = \frac {3\pi} 2 +\epsilon & \vartheta = 2\pi\\
		\includegraphics[scale=0.3]{I1_pi.eps} & \includegraphics[scale=0.3]{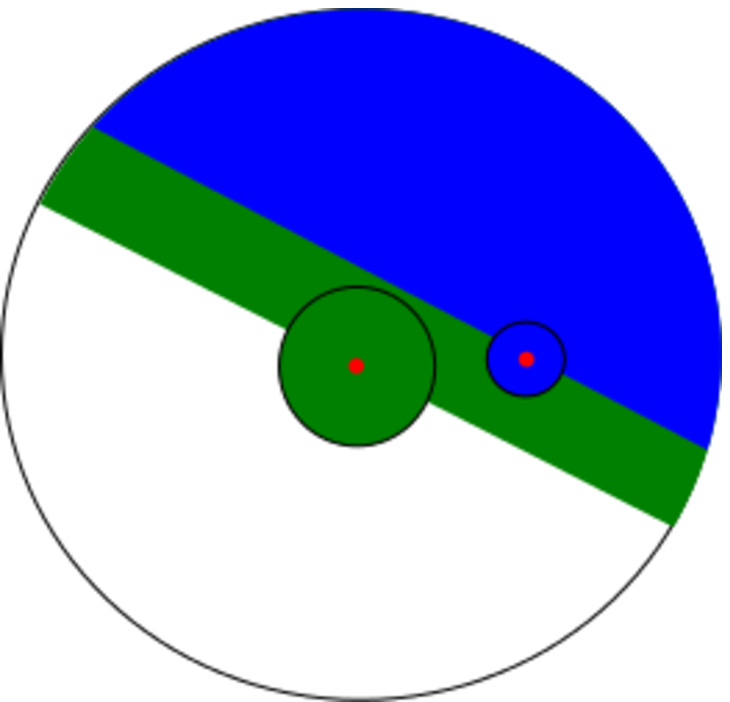} & \includegraphics[scale=0.3]{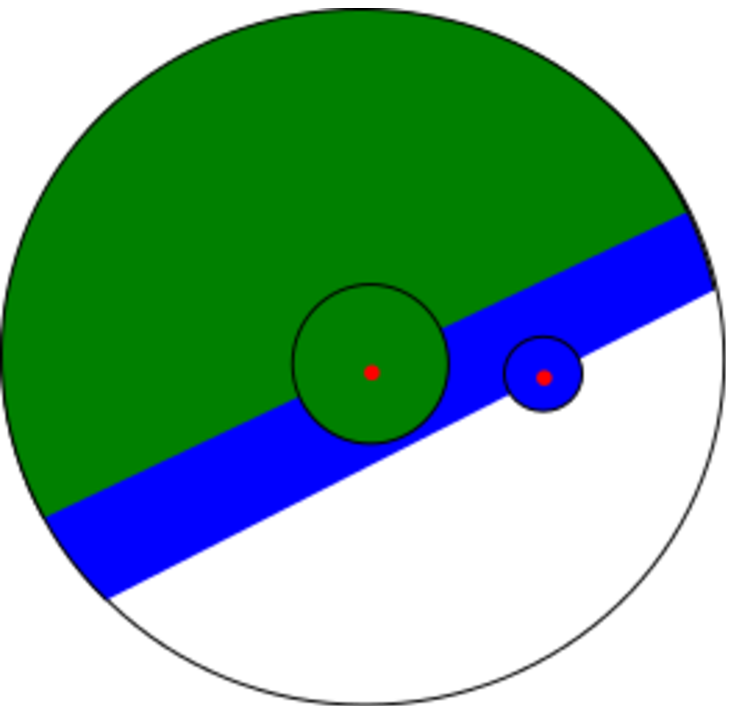} & \includegraphics[scale=0.3]{I1_0.eps}
	 	 }
	\end{xy}$$
	
In this situation we can compute the first cohomology groups $H^1 \left( \overline A \times \{\vartheta\}, \overline{\beta}^{\vartheta}_{\psi, !} \mathcal K^{\vartheta}\right)$ in another way, namely by using \u{C}ech cohomology. \\
Therefore consider the following curves $\alpha_1, \alpha_2, \alpha_3, \alpha_4$ in $\overline A$: 
	\begin{center}
		\includegraphics[scale=0.3]{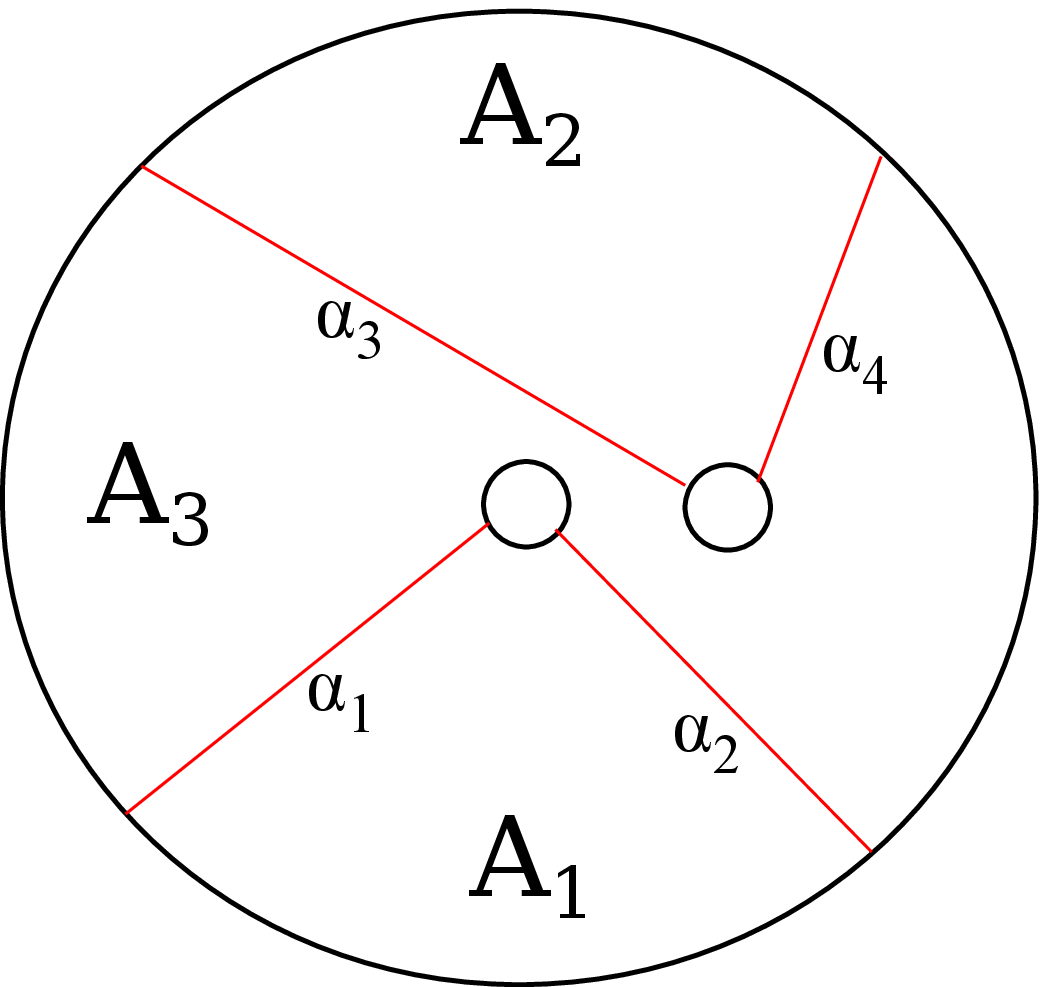} 
	\end{center}
Then for $\vartheta_0 \in [0, \frac{\pi} 2)$ the closed covering $\mathfrak A = A_1 \cup A_2 \cup A_3$ of $\overline A$ defines a Leray covering of  $\beta^{\vartheta_0}_{\psi, !} \mathcal K^{\vartheta_0}$ and for $\beta^{\vartheta_1}_{\psi, !} \mathcal K^{\vartheta_1}$, whereby $\vartheta_1 := \vartheta_0 + \pi$.  This can be proved easily by using the exact sequence of Chapter 3.3.4 and doing the same calculations for the restriction to the intersections $\alpha_k$. \\
For $\vartheta_0  \in [\frac {\pi} 2, \pi)$ the following curves define a Leray covering for  $\beta^{\vartheta_0}_{\psi, !} \mathcal K^{\vartheta_0}$ and $\beta^{\vartheta_1}_{\psi, !} \mathcal K^{\vartheta_1}$:
	\begin{center}
		\includegraphics[scale=0.3]{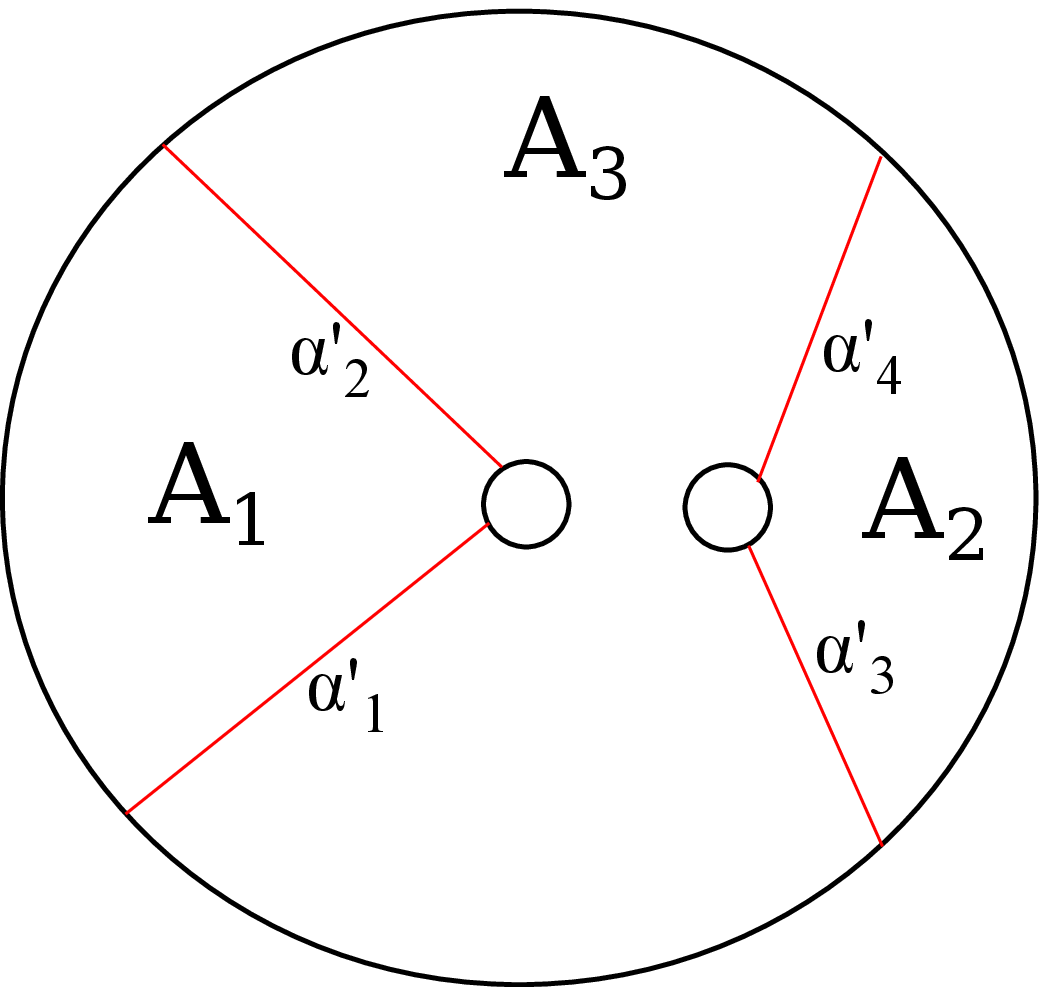} 
	\end{center}
Thus we can conclude: 
\begin{lemma} For every pair of angles $ \left( \vartheta_0, \vartheta_1:= \vartheta_0+ \pi\right)$ of $\mathbb S^1$ the construction above defines a closed covering $\mathfrak A$ of $\overline A$, such that $\mathfrak A$ is a common Leray covering of  $\beta^{\vartheta_i}_{\psi, !} \mathcal K^{\vartheta_i}$ ($i=1,2$ and $\psi= 0, \frac 1 t$).  Consequently we get an isomorphism 
\[H^1 \left( \overline A \times \{\vartheta_i\}, \overline{\beta}^{\vartheta_i}_{\psi,!} \mathcal K^{\vartheta_i}\right) \rightarrow \check H^1 \left( \mathfrak A,\overline{\beta}^{\vartheta_i}_{\psi, !} \mathcal K^{\vartheta_i}\right)\]
\end{lemma}   

In the following lemma we will compute the cohomology groups concretely for $\vartheta=0$ and $\vartheta= \pi$. 

\begin{lemma} \label{LC} 
	\begin{itemize}
		\item $H^1 \left( \overline A \times \{0\}, \overline{\beta}_{0, !}^{0} \mathcal K^{0}\right) \cong \mathcal K^0_{x_1}$
		\item $H^1 \left( \overline A \times \{0\},  \overline{\beta}_{0, !}^{\frac 1 t} \mathcal K^{0}\right)=H^1 \left( \overline A \times \{0\},  \overline{\beta}^0_ ! \mathcal K^{0}\right) 			\cong \mathcal K^0_{x_1}\oplus \mathcal K^0_{x_3}$
		\item $H^1 \left( \overline A \times \{\pi\},  \overline{\beta}_{\pi, !}^{\frac 1 t} \mathcal K^{\pi}\right) \cong \mathcal K^{\pi}_{x_4}$ 
		\item $H^1 \left( \overline A \times \{\pi\},  \overline{\beta}_{\pi, !}^{0} \mathcal K^{\pi}\right) =H^1 \left( \overline A \times \{\pi\},  \overline{\beta}^{\pi}_! \mathcal K^{\pi}\right) 			\cong \mathcal K^{\pi}_{x_2} \oplus \mathcal K^{\pi}_{x_4}$
	\end{itemize}
\end{lemma}

\begin{proof}
Just notice that $\check C^0 =0$ and thus $\check H^1= \check C^1$.
\end{proof}


\subsection {Stokes data associated to $\mathcal L$ }

Using the functor constructed in \cite{HS} we can associate a set of Stokes data to the Stokes-filtered local system $\mathcal L$ described above. 

\begin{construction} \label{SaHe} Fix two intervals 
$$I_0=  \left( 0-\epsilon, \pi + \epsilon\right), \: \: I_1= \left( -\pi -\epsilon, 0 + \epsilon\right)$$
of length $\pi + 2 \epsilon$ on $\mathbb  S^1$, such that the intersection $I_0 \cap I_1$ consists of $ \left( 0- \epsilon, 0 + \epsilon\right)$ and $ \left( \pi -\epsilon, \pi + \epsilon\right)$. Observe, that the intersections do not contain the Stokes directions $- \frac {\pi} 2, \frac {\pi} 2$. To our given local system $\mathcal L$ we associate: 
	\begin{itemize}
		\item a vector space associated to the angle $0$, i.\:e. the stalk $\mathcal L_0$. It comes equipped with the Stokes filtration.
		\item a vector space associated to the angle  $\pi$, i.\:e. the stalk $\mathcal L_{\pi}$, coming equipped with the Stokes filtration.
		\item vector spaces associated to the intervals $I_0$ and $I_1$, i.\:e. the global sections $\Gamma \left( I_0, \mathcal L\right)$, $ \Gamma \left( I_1, \mathcal L\right)$
		\item a diagram of isomorphisms (given by the natural restriction to the stalks): 
		\[
		\begin{xy}
  			\xymatrix
  			{
  			&  \Gamma \left( I_0, \mathcal L\right)  \ar[dl]_{a_0} \ar[rd]^{a'_0}& \\
  			\mathcal L_0 & & \mathcal L_{\pi}\\
			&  \Gamma \left( I_1, \mathcal L\right)  \ar[ul]^{a_1} \ar[ru]_{a'_1}&
 	 		}
		\end{xy}
		\]
	\end{itemize}
The filtrations on the stalks are opposite with respect to the maps $a'_0 a_0^{-1}$ and $a_1 {a'}_1^{-1}$, i.\:e.
 \[ \mathcal L_0 = \bigoplus_{\phi \in \left\{0 , \frac 1 t \right\}} L_{\leq \phi, 0} \cap a'_0 a_0^{-1} \left(L_{\leq \phi, \pi}\right), \hspace*{0.5cm}
 \mathcal L_{\pi} = \bigoplus_{\phi \in \left\{0 , \frac 1 t \right\}} L_{\leq \phi, \pi} \cap a_1 {a'}_1^{-1} \left(L_{\leq \phi, 0}\right)  \]
Furthermore we know, by using the isomorphisms $\Omega$ and $\Gamma$, that $\mathcal L_0 \cong \mathcal K_{x_1} \oplus \mathcal K_{x_3}$ whereby $\mathcal L_{\leq 0, 0} \cong \mathcal K_{x_1}$. Thus we have a splitting $\mathcal L_0 = G_0 \oplus G_{\frac 1 t}$ with $G_0 = \mathcal L_{\leq 0, 0}$. The same holds for $\mathcal L_{\pi}$: $\mathcal L_{\pi} = H_0 \oplus H_{\frac 1 t}$ with $H_{\frac 1 t}= \mathcal L_{\leq \frac 1 t , \pi}$. \\

Thus the set of data $ \left( G_0, G_{\frac 1 t}, H_0, H_{\frac 1 t}, S_0^{\pi}, S_{\pi}^0\right)$ with: 
	\begin{itemize}
		\item $ \mathcal L_0 = G_0 \oplus G_{\frac 1 t}$ and $\mathcal L_\pi = H_{\frac 1 t} \oplus H_0$
		\item $S_0^{\pi}:  G_0 \oplus G_{\frac 1 t}  \xrightarrow{ a_0^{-1}}  \Gamma \left( I_0, \mathcal L \right)\xrightarrow{ a'_0}  H_0 \oplus H_{\frac 1 t}$
		\item $S_{\pi}^0: H_{\frac 1 t} \oplus H_{0}  \xrightarrow{ a_1^{-1}}  \Gamma \left( I_1, \mathcal L \right)\xrightarrow{ a'_1}  G_{\frac 1 t} \oplus G_{0}$
	\end{itemize} 
describes a set of Stokes data associated to the local system $\mathcal L$.\\
By exhaustivity of the filtrations on $\mathcal L_{\vartheta}$ ($\vartheta = 0, \pi$) the isomorphism $\Omega$ induces  
$$\mathcal L_{\vartheta} \cong H^1 \left( \overline A \times \vartheta, \overline{\beta}^{\vartheta}_! \mathcal K^{\vartheta}\right)$$
 where $\overline{\beta}^{\vartheta}: B^{\vartheta} \hookrightarrow \overline A \times {\vartheta}$ corresponds to the $\overline{\beta}^{\vartheta}_{\psi}$ with $\psi \geq_{\vartheta} \phi$ for  $ \psi, \phi \in \Phi=\{ 0,\frac 1 t \}$ and $B^{\vartheta}:= B_{\psi}^{\vartheta} \supset B_{\phi}^{\vartheta}$.\\
In the same way, for an open interval $I \subset \mathbb S^1$ let $\mathcal K^I$ be the restriction of $\mathcal K$ to $\overline A \times I$ and define $\overline{\beta}^{I}:B^I \hookrightarrow \overline A \times I$ the inclusion of the subspace $B^I$, which is the support of $\operatorname{DR}^{mod \:  D} (\mathcal M \otimes \mathcal E^{\frac 1 y})$ according to Remark \ref{remarknew}. Notice  that ${\overline{\beta}^{I}} _{| \overline A \times {\vartheta}}= \overline{\beta}^{\vartheta}$. Then, by the isomorphism $\Omega$, we identify 
$$\Gamma \left( I, \mathcal L\right) \cong H^1 \left(  \overline A \times I, \overline{\beta}^I_! \mathcal K^I\right).$$
With these isomorphisms we have restriction morphisms (according to the restrictions to the stalks)
\[ \rho_{\vartheta}: H^1 \left(  \overline A \times I_0, \overline{\beta}^{I_0}_! \mathcal K^{I_0} \right)\xrightarrow{\cong} H^1 \left( \overline A \times \vartheta, \overline{\beta}^{\vartheta}_! \mathcal K^{\vartheta}\right) \text{ for } \vartheta \in I_0\]
\[ \rho'_{\vartheta}: H^1 \left(  \overline A \times I_1, \overline{\beta}^{I_1}_! \mathcal K^{I_1} \right)\xrightarrow{\cong} H^1 \left( \overline A \times \vartheta, \overline{\beta}^{\vartheta}_! \mathcal K^{\vartheta}\right) \text{ for } \vartheta \in I_1\]
\end{construction}

This yields to a new way of describing Stokes data associated to the local system $\mathcal L$:
\begin{theorem} \label{Thm1} Fix the following data:
	\begin{itemize}
		\item vector spaces $ L_0:= H^1 \left( \overline A \times \{0\}, \overline{\beta}^0_! \mathcal K^0\right)$ and $L_1:= H^1 \left( \overline A \times \{\pi\}, \overline{\beta}^{\pi}_! 			\mathcal K^{\pi}\right)$
		\item morphisms $\sigma_0^{\pi}:= \rho_{\pi} \circ \rho_0^{-1}$ and $\sigma_{\pi}^0:=  \rho'_{0} \circ {\rho'}_{\pi}^{-1}$ , with $\rho_{\vartheta}$ and $\rho'_{\vartheta}$ 				defined as above. 
	\end{itemize}
Then
$$ \left( L_0, L_1, \sigma_0^{\pi}, \sigma_{\pi}^0\right)$$ 
defines a set of Stokes data for $\mathcal H^0 p_+  \left( \mathcal M \otimes \mathcal E^{\frac 1 y}\right)$.
\end{theorem}

\begin{proof}
The isomorphism $\Omega: \mathcal L_{\vartheta} \to H^1 \left( \overline A \times \{\vartheta\}, \overline{\beta}^{\vartheta}_! \mathcal K^{\vartheta}\right)$ passes the filtrations on $\mathcal L_0$ and $\mathcal L_{\pi}$ to $L_0$ and $L_1$ and therefore yields to a suitable graduation of $L_0 = \widetilde G_0 \oplus \widetilde G_{\frac 1 t} \cong \mathcal K^0_{x_1} \oplus \mathcal K^0_{x_3}$ and $L_1 = \widetilde H_{\frac 1 t} \oplus \widetilde H_{0} \cong \mathcal K^{\pi}_{x_4} \oplus \mathcal K^{\pi}_{x_2}$. \\
On the level of graduated spaces $\Omega: G_0 \oplus G_{\frac 1 t} \to \widetilde G_0 \oplus \widetilde G_{\frac 1 t}$ (resp. $H_{\frac 1 t} \oplus H_0 \to \widetilde H_{\frac 1 t}  \oplus  \widetilde H_0$)  is obviously described by a block diagonal matrix. 
Furthermore,  by definition of $\rho$ and $\rho'$, the maps $\sigma_0^{\pi}$ and $\sigma_{\pi}^0$ can be read as  $\sigma_0^{\pi} = \Omega \circ S_0^{\pi} \circ \Omega^{-1}$ and $\sigma_{\pi}^0= \Omega \circ S_{\pi}^0 \circ \Omega^{-1}$. Since, according to the construction in \ref{SaHe},  $S_0^{\pi}$ is upper block triangular (resp. $S_{\pi}^0$ is lower block triangular) the same holds for $\sigma_0^{\pi}$ (resp. $\sigma_{\pi}^0$).  
\end{proof} 


\subsection{Explicit computation of the Stokes matrices}

The determination of the Stokes data thus corresponds to the following picture: \\

	\[
	\begin{xy}
	\xymatrixrowsep{0.1cm}
	\xymatrix
  		{
  		 \includegraphics[scale=0.13]{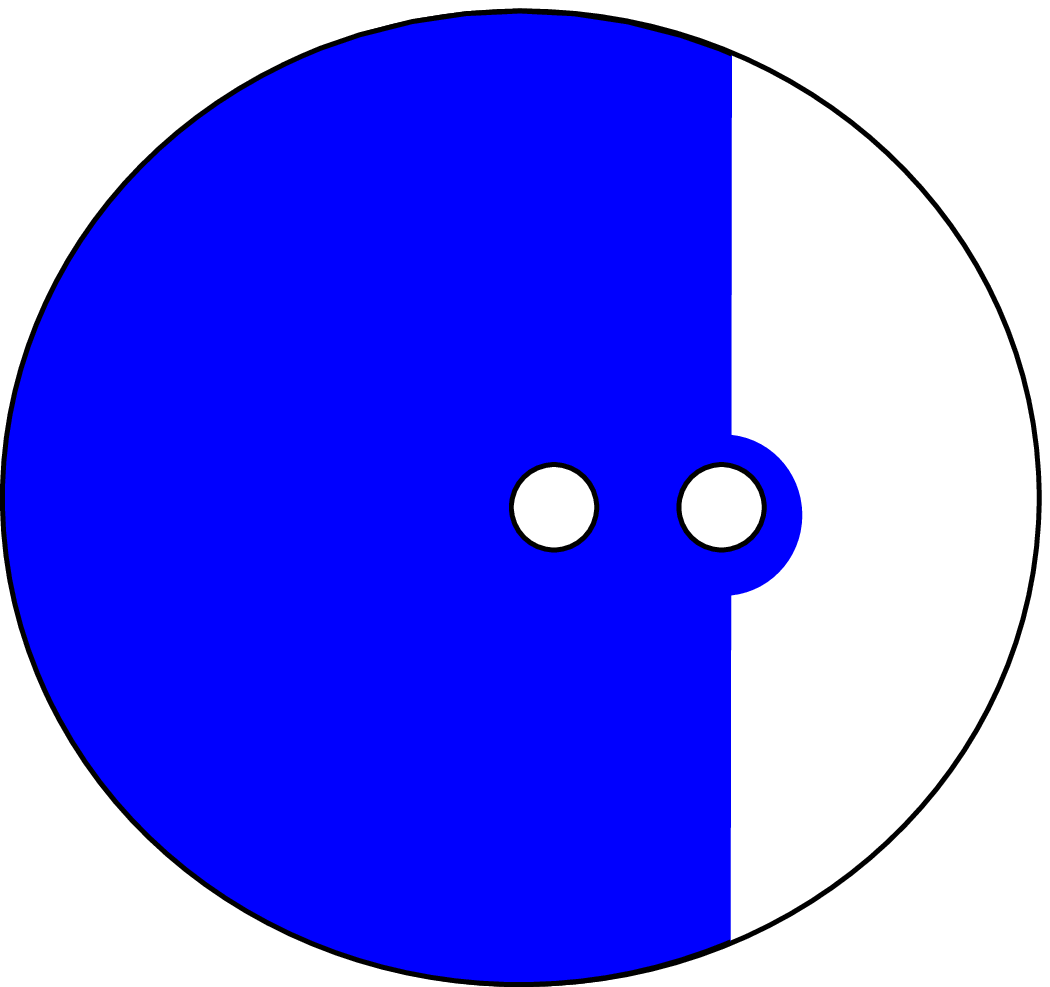}  & \includegraphics[scale=0.3]{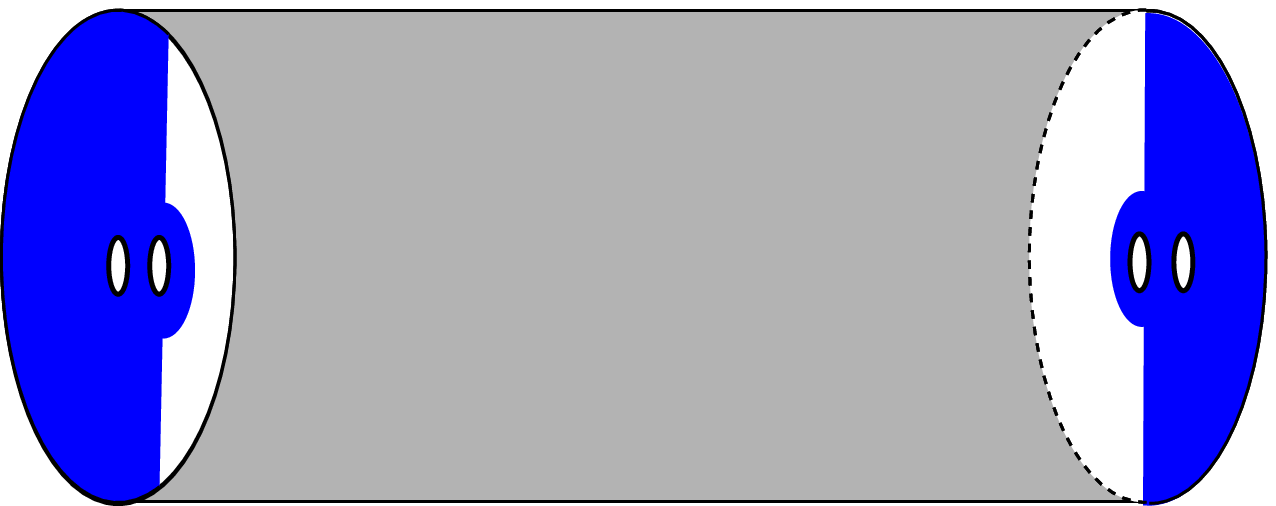} & \includegraphics[scale=0.15]{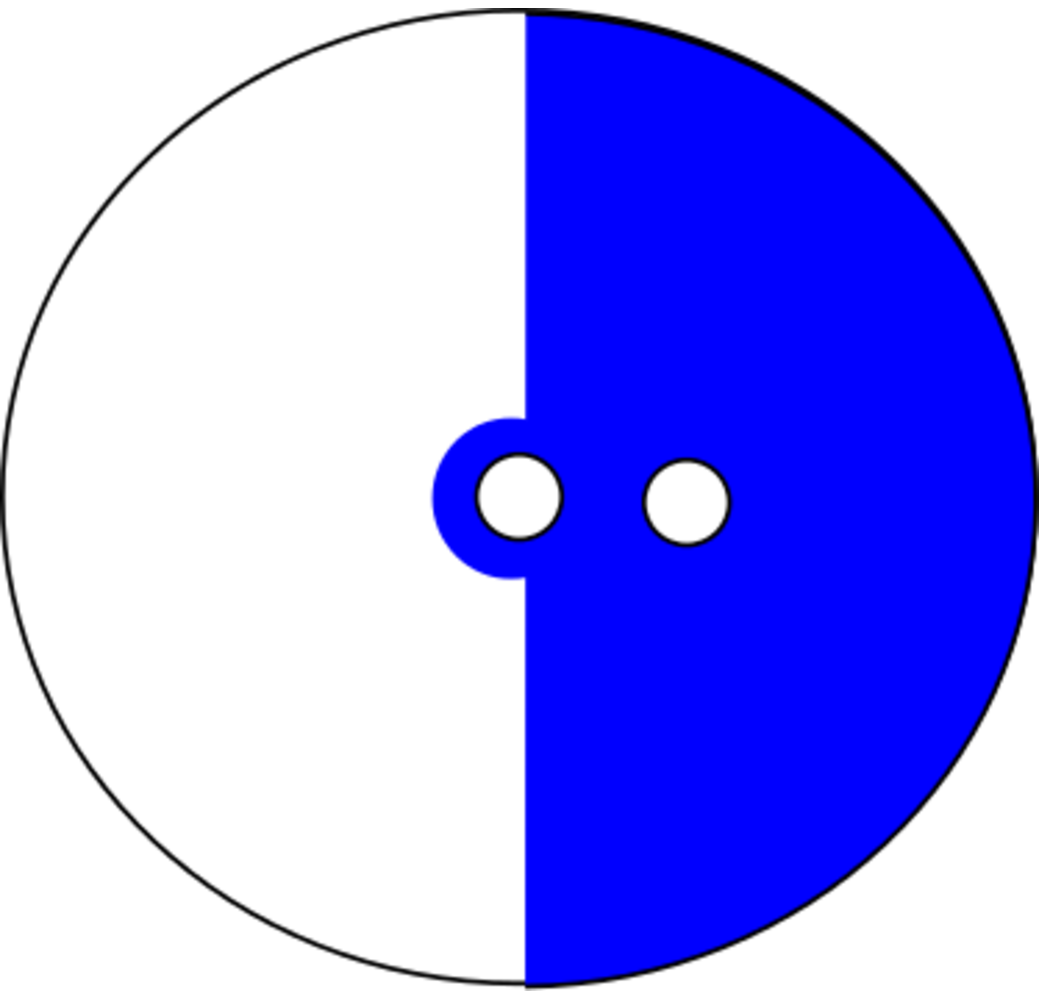} \\
								 &H^1 \left( \overline A \times I_0, \overline{\beta}^{I_0}_! \mathcal K \right)\ar[ldd]_{\cong} \ar[rdd]^{\cong}&\\
								&& \\
		 H^1 \left( \overline A \times \{ 0\}, \overline{\beta}^0_! \mathcal K^0 \right) \ar@/ ^ 0.5 cm /[rr] & &H^1 \left( \overline A \times \{ \pi\}, \overline{\beta}^{\pi}_! \mathcal K^{\pi}		\right)\ar@/^ 0.5cm/[ll]\\
		 &&\\
		 &H^1 \left( \overline A \times I_1, \overline{\beta}^{I_1}_! \mathcal K \right)\ar[luu]^{\cong} \ar[ruu]_{\cong}&
 	 	}
	\end{xy}
	\]
\vspace*{0.5cm}\\
In Lemma \ref{LC} we have already computed the cohomology groups for $\vartheta=0$ and $\vartheta= \pi$ using the Leray covering $\mathfrak A$.
Now for $l=0,1$ we fix a diffeomorphism $\overline A \times I_l \xrightarrow {\sim} \overline A \times I_l$ by lifting the vector field $\partial_{\vartheta}$ to $\overline A \times \mathbb S^1$ such that the lift is equal to $\partial_{\vartheta}$ away from a small neighborhood of $\partial \overline A$ and such that the diffeomorphism induces $B^{I_l} \xrightarrow{\cong} B^{\vartheta_{l+1}} \times I_l$ where $\vartheta_1= \pi$ and $ \vartheta_2= 0$. It induces a diffeomorphism $\overline A \times \{ \vartheta_l\} \xrightarrow{\cong} \overline A \times \{ \vartheta_{l+1}\}$ and an isomorphism between the push forward of $\mathcal K^{\vartheta_l}$ and $\mathcal K^{\vartheta_{l+1}}$. Moreover it  sends the boundary $\partial B^{\vartheta_l}$ to $\partial B^{\vartheta_{l+1}}$ (i.\:e. the boundaries are rotated in counter clockwise direction by the angle $\pi$). Via this diffeomorphism the curves $\alpha_i \subset \overline A \times {\vartheta_l}$ are sent to curves $\widetilde{\alpha}_i \subset \overline A \times {\vartheta_{l+1}}$ and therefore induce another Leray covering $\widetilde{\mathfrak A}$ of $H^1 \left( \overline A \times \{\vartheta_{l+1}\}, \overline{\beta}^{\vartheta_{l+1}}_! \mathcal K^{\vartheta_{l+1}}\right)$.
	\begin{center}
		\includegraphics[scale=0.4]{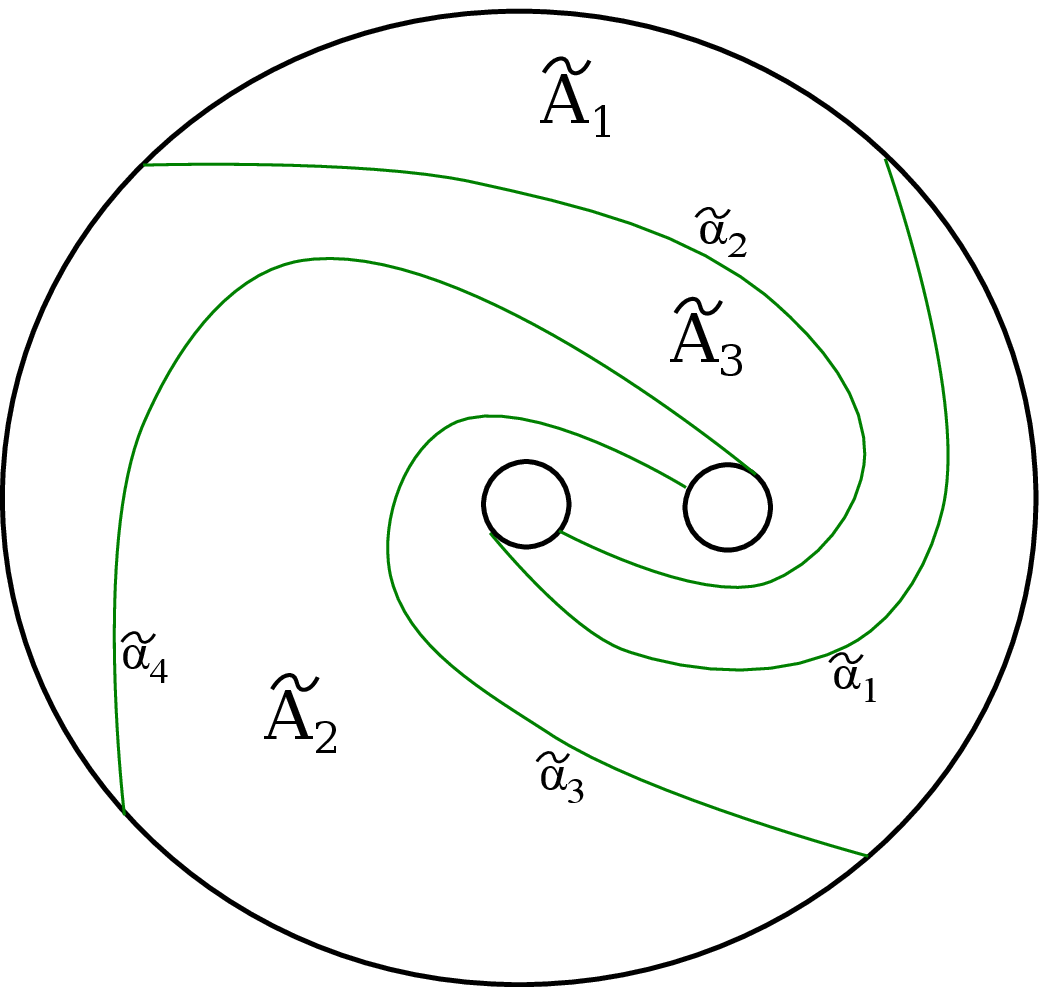} 
	\end{center}
Explicitly, for $l=0$ we get a Leray covering $\widetilde{\mathfrak A}$ of $H^1 \left( \overline A \times \{\pi\}, \overline{\beta}^{\pi}_! \mathcal K^{\pi}\right)$ and, as in the previous chapter, we get an isomorphism 
$\widetilde {\Gamma}_{\pi}: H^1 \left( \overline A \times \{\pi\}, \overline{\beta}^{\pi}_ ! \mathcal K^{\pi}\right) \to \check H^1 \left( \widetilde{\mathfrak A}, \overline{\beta}^{\pi}_ ! \mathcal K^{\pi}\right)$, which gives us: 
\[H^1 \left( \overline A \times \{\pi\}, \overline{\beta}^{\pi}_ ! \mathcal K^{\pi}\right) \cong \mathcal K^{\pi} \left(  \widetilde \alpha_1\right) \oplus \mathcal K^{\pi} \left(  \widetilde \alpha_3\right) \cong \mathcal K^{\pi}_{x_1} \oplus \mathcal K^{\pi}_{x_3} \]
Thus the above diffeomorphism leads to an isomorphism
\[\mu_0^{\pi}: \mathcal K^0_{x_1} \oplus  \mathcal K^0_{x_3} \xrightarrow{\cong} \mathcal K^{\pi}_{x_1} \oplus  \mathcal K^{\pi}_{x_3} \]
Furthermore let us fix the following vector space $\mathbb V := \mathcal K^{\pi}_c$ to be the stalk of the local system $\mathcal K$ at the point $c:=  \left( 0, \pi, \frac 1 2 , 0\right)$ (which is obviously a point in the fiber $\overline A \times \{ \pi\}$). By analytic continuation we can identify  every non-zero stalk $\mathcal K^{\vartheta}_x$ with $\mathbb V$ for all $\vartheta$.\\
Now consider the following diagram of isomorphisms: 

	\[
	\begin{xy}
		\xymatrixrowsep{1.5cm}
		\xymatrixcolsep{0.5cm}
		\xymatrix
  		{
  		 \mathcal L_0 \ar[d] _{\cong} \ar[rrr]^{S_0^{\pi}}  &  & & \mathcal L_{\pi} \ar[d]^{\cong}\\
		 H^1 \left( \overline A \times \{ 0\}, \overline{\beta}^0_! \mathcal K^0 \right) \ar[d] _{\Gamma_0} \ar[rrr] &  & &H^1 \left( \overline A \times \{ \pi\}, \overline{\beta}^{\pi}_! \mathcal 		K^{\pi}\right) \ar[ld]_{\widetilde{\Gamma}_{\pi}}  \ar[d] ^{\Gamma_{\pi}} \\
		 \mathcal K^0_{x_1} \oplus  \mathcal K^0_{x_3} \ar[rr]^{\mu_0^{\pi}} \ar[d] &&\mathcal K^{\pi}_{x_1} \oplus  \mathcal K^{\pi}_{x_3} \ar[r]^{\nu_{\pi}} \ar[d] &\mathcal K^{\pi}		_{x_2} \oplus  \mathcal 	K^{\pi}_{x_4} \ar[d] \\
		 \mathbb V \oplus \mathbb V \ar[rr]^{Id} && \mathbb V \oplus \mathbb V \ar[r]^{N_{\pi}} & \mathbb V \oplus \mathbb V
 	 	}
	\end{xy}
	\]
\vspace*{0.3cm}
\enlargethispage{1cm}
It remains to determine the map $\nu_{\pi}$  (respectively $N_{\pi}$). Therefore we will combine the coverings $\mathfrak A$  and $\widetilde {\mathfrak A}$ of $\overline A \times \{\pi\}$ to a refined covering $\mathfrak B$. We get refinement maps ${\operatorname{ref}}_{\mathfrak A \to \mathfrak B}$ and ${\operatorname{ref}}_{\widetilde{\mathfrak A} \to \mathfrak B}$ and receive the following picture:

	\[
	\begin{xy}
		\xymatrixrowsep{0.5cm}
		\xymatrixcolsep{-0.5cm}
		\xymatrix
  		{
  		 \includegraphics[scale=0.3]{Axvartheta.eps} \ar[rd]_{{\operatorname{ref}}_{\mathfrak A \to \mathfrak B}}   &&&&    \includegraphics[scale=0.3]{Axpitilde.eps}  \ar[ld] 			_{{\operatorname{ref}}_{\widetilde{\mathfrak A} \to \mathfrak B}}  \\
		 &&&&\\
		 &&&&\\
		 &	& \includegraphics[scale=0.4]{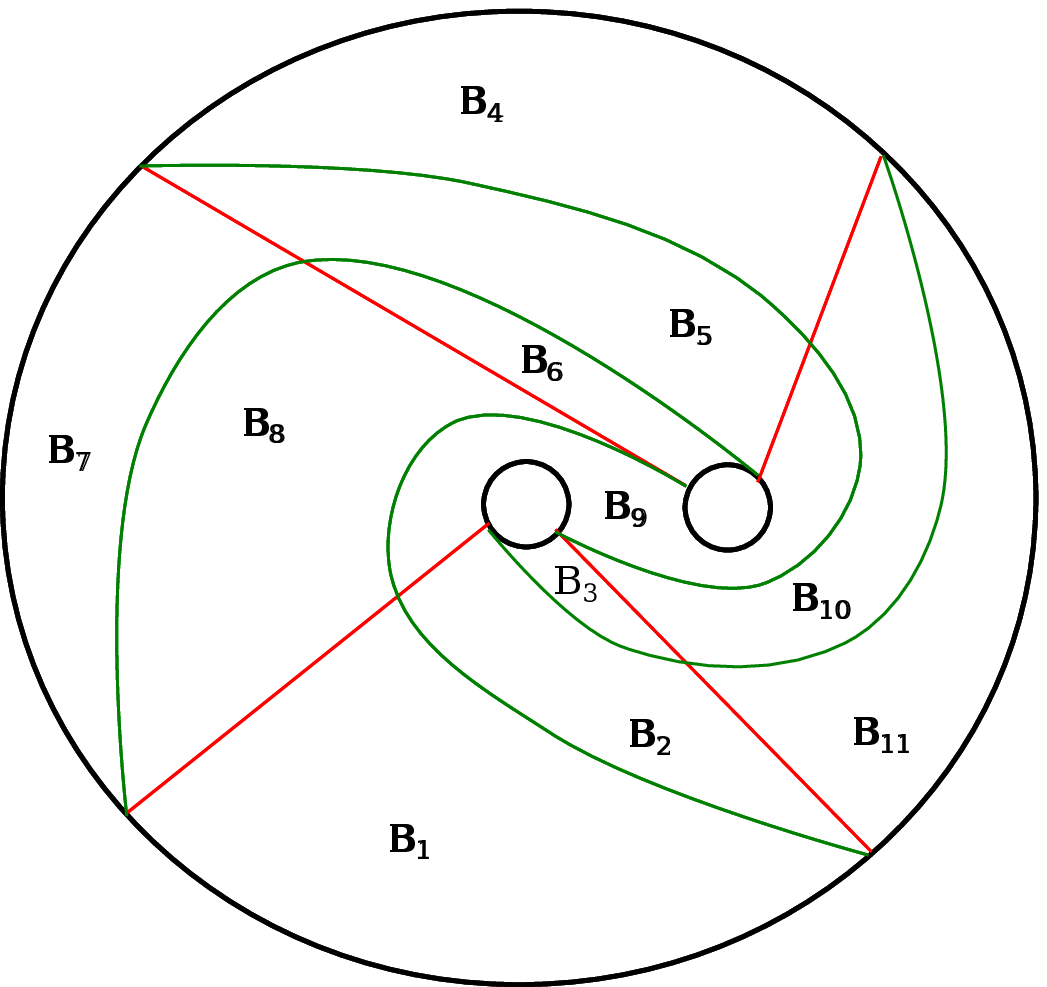} & &
		}
	\end{xy}
	\]

Since $\mathfrak B$ is again a Leray covering the refinement maps induce isomorphisms on the cohomology groups. 
Thus we will extend' the above diagram of isomorphisms in the following way: 
	\[
	\begin{xy}
		\xymatrixrowsep{1.3cm}
		\xymatrixcolsep{0.6cm}
		\xymatrix
  		{
  		&\check H^1 \left( \mathfrak A, \overline{\beta}^{\pi}_ ! \mathcal K^{\pi}\right) \ar[r]^{\nu_{\pi}'}&\check H^1 \left( \mathfrak A, \overline{\beta}^{\pi}_ ! \mathcal K^{\pi}\right)&\\
		 H^1 \left( \overline A \times \{ \pi\}, \overline{\beta}^0_! \mathcal K^0 \right)\ar[r]^{\cong} \ar[rd] _{\widetilde{\Gamma}_{\pi}} & \check H^1 \left( \widetilde{\mathfrak A}, 			\overline{\beta}^{\pi}_ ! \mathcal K^{\pi}\right) \ar[d]_{\cong} \ar[u]^{{\operatorname{ref}}_{\widetilde{\mathfrak A} \to \mathfrak B}}  & \check H^1 \left( \mathfrak A, 				\overline{\beta}^{\pi}_ ! \mathcal K^{\pi}\right) \ar[d]^{\cong} \ar[u]_{{\operatorname{ref}}_{\mathfrak A \to \mathfrak B}} &H^1 \left( \overline A \times \{ \pi\}, \overline{\beta}		^{\pi}_! \mathcal K^{\pi}\right) \ar[l]_{\cong} \ar[ld]^{\Gamma_{\pi}} \\
		& \mathcal K^{\pi}_{x_1} \oplus  \mathcal K^{\pi}_{x_3} \ar[r]^{\nu_{\pi}} \ar[d]_{\cong}  &\mathcal K^{\pi}_{x_2} \oplus  \mathcal K^{\pi}_{x_4} \ar[d]^{\cong}& \\
		& \mathbb V \oplus \mathbb V \ar[r]^{N_{\pi}} & \mathbb V \oplus \mathbb V&
 	 	}
	\end{xy}
	\]	

To determine the refinement maps on the first cohomology groups we have to consider the \u{C}ech complex for $\mathfrak B$.\\
\\
We set the following index sets
	\begin{itemize}
		\item $ I :=\{1, 2, 3\}$ (indices corresponding to the covering $\mathfrak A= \bigcup_{i \in I} A_i$)
		\item $\widetilde I := \{ 1,2,3\}$ (corresponding to $\widetilde{\mathfrak A}= \bigcup_{i \in \widetilde I} \widetilde A_i$)
		\item $J:=\{1,2, 3, \ldots , 11\}$ (corresponding to ${\mathfrak B}= \bigcup_{j \in J} B_j$)
		\item $K:= I \times J$
		\item $J':=\{ 2,3,9,10,11\} \subset J$
		\item $K'= \{ \left( 1,2\right),  \left( 1,9\right),  \left( 1,11\right),  \left( 2,3\right),  \left( 2,8\right),  \left( 2,9\right),  \left( 2,10\right),  \left( 2,11\right),  \left( 3,9\right), \left( 3,10\right),  			\left( 3,11\right),\\ 
			\hspace*{1.0cm}   \left( 4,9\right),  \left( 4,10\right),  \left( 4,11\right),  \left( 5,9\right),  \left( 5,10\right),  \left( 6,9\right),  \left( 8,9\right),  \left( 9,10\right),  \left( 10,11\right) \} 			\subset K$
		\end{itemize}
\vspace*{0.5cm}
The \u{C}ech complex for $\mathfrak B$ is given by: 
	\[
	\begin{xy}
		\xymatrixrowsep{0.2cm}
		\xymatrixcolsep{1.0cm}

		\xymatrix
		{
		\check C^0 \ar@{=}[d]  \ar[r]^{d_0}  & \check C^1\ar@{=}[d]\ar[r]^{d_1} & \check C^2   \ar[r]^{d_2} & \ldots&\\
		 \bigoplus\limits_{j \in J'  }\check H^0  \left( B_j, \overline{\beta}^{\pi}_! \mathcal K^{\pi}\right) &\bigoplus\limits_{ \left( i,j\right) \in K' }\check H^0  \left( B_i \cup B_j, 				\overline{\beta}^{\pi}_! \mathcal K^{\pi}\right)& &&
		}
	\end{xy}
	\]

\vspace*{0.5cm}
To determine the map $d_0$ we will use the following identification of $\check C^1$: \\
As above, by analytic continuation we can identify each component $\check H^0 \left( B_i \cup B_j, \overline{\beta}^{\pi}_! \mathcal K^{\pi}\right)$ with $\mathbb V^{k \left( i,j\right)}:=\bigoplus _{l=1}^{k \left( i,j\right)}\mathbb V$ where $k \left( i,j\right)$ denotes the number of connected components of $ \left( B_i \cap B_j\right)$. $k \left( i,j\right)= 1$ for all $ \left( i,j\right)\in J$ except for $ \left( 3,9\right)$ and $ \left( 6,9\right)$ where it is equal to $2$. So we have an isomorphism 
\[  \check C^1 \cong \bigoplus_{ \left( i,j\right) \in K'} \mathbb V^{k \left( i,j\right)}\]
\vspace*{0.5cm}

Now if we take a section $b_9$ of $B_9$, restrict it to a boundary component of $B_9$ (which is the intersection with one of the bordering $B_i$s) and identify it with $\mathbb V$, we have to take care about the monodromies $S, T$ around the two leaks in $\overline A \times \{ \pi\}$. The following picture shows, by restriction to which boundary component we receive monodromy. 

	\begin{center}
		\includegraphics[scale=0.4]{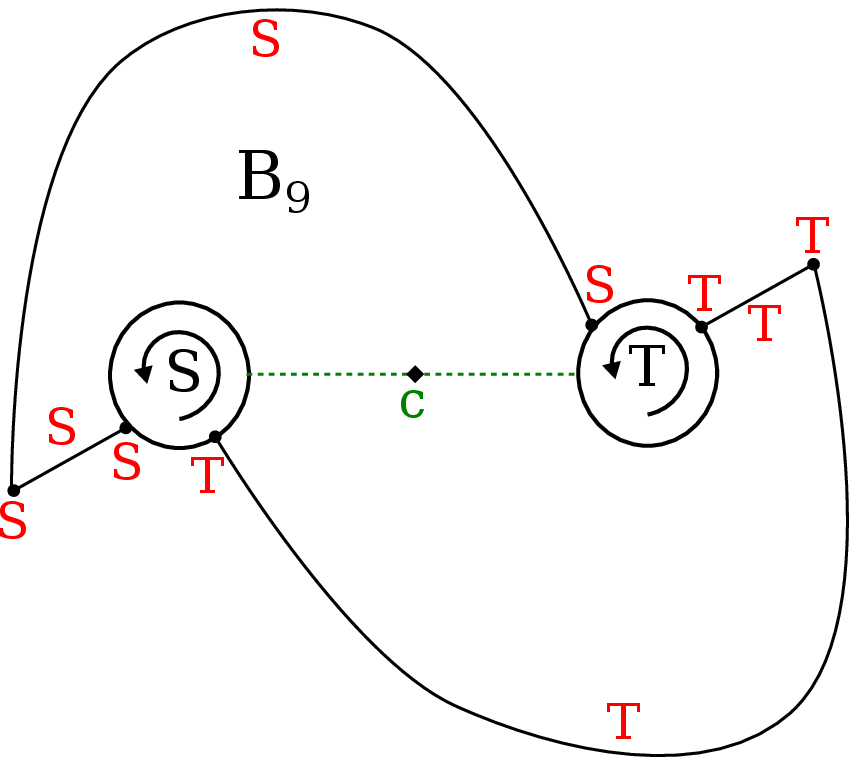} 
	\end{center}

\begin{remark} 
If we follow $S$ and $T$  in the coordinates of $\overline A \times \{ \vartheta \}$, we can also describe them in terms of monodromy around the divisor components: 
$S$ can be described by a path $\gamma_S: [0,1] \to \overline A \times \{ \vartheta\} , \tau \mapsto  \left( 0, \vartheta, 0,  \tau \cdot 2 \pi\right)$ in the coordinates $ \left( |t |, \vartheta, |x|, \theta_x\right)$. If we look at this path in the ( complex) blow up of the singular locus of $\mathcal M$, $\gamma_S$ corresponds to the monodromy around $\widetilde S_1$. In the same way $T$ is given by  $\gamma_T: [0,1] \to \overline A \times \{ \vartheta\} , \tau \mapsto  \left( 0, \vartheta, 0,  \tau \cdot 2 \pi\right)$ in the coordinates $ \left( |\tilde u'_1 |, \vartheta, |\tilde v'_1|, \theta_{\tilde v'_1}\right)$ and it corresponds to the monodromy around the strict transform of $S_1$. Note that $S$ and $T$ do not depend on $\vartheta$.
\end{remark}

Now one can write down easily the map $d_0$ of the \u{C}ech complex  and the refinement maps $\operatorname{ref}_{\mathfrak A \to \mathfrak B}$ respectively $\operatorname{ref}_{\widetilde{\mathfrak A} \to \mathfrak B}$ with respect to the above analytic continuation. 

Furthermore let $ \left( a_2, a_4\right)$ be a base of $\check H^1 \left( \mathfrak A, \overline{\beta}_!^{\pi} \mathcal K^{\pi}\right) \cong \mathcal K^{\pi}_{x_2} \oplus \mathcal K^{\pi}_{x_4} \cong \mathbb V \oplus \mathbb V$ and $ \left( \tilde a_1, \tilde a_3\right)$ a base of $\check H^1 \left( \widetilde{\mathfrak A}, \overline{\beta}_!^{\pi} \mathcal K^{\pi}\right) \cong \mathcal K^{\pi}_{x_1} \oplus \mathcal K^{\pi}_{x_3} \cong \mathbb V \oplus \mathbb V$. We receive two bases of $\check H^1 \left( \mathfrak B, \overline{\beta}_!^{\pi} \mathcal K^{\pi}\right)$, namely  $\operatorname{ref}_{\widetilde{\mathfrak A} \to \mathfrak B}  \left( \tilde a_1, \tilde a_3\right)$ and $\operatorname{ref}_{\mathfrak A \to \mathfrak B}  \left( a_2, a_4\right)$. Thus $\nu'$ is determined by representing the base  $\operatorname{ref}_{\widetilde{\mathfrak A} \to \mathfrak B}  \left( \tilde a_1, \tilde a_3\right)$ in terms of 
$\operatorname{ref}_{\mathfrak A \to \mathfrak B}  \left( a_2, a_4\right)$. \\
As before we identify $\check C^1 \left( \mathfrak B,  \overline{\beta}_!^{\pi} \mathcal K^{\pi}\right) \cong \bigoplus \mathbb V^{k \left( i,j\right)}$ and we end up in solving the following equation for each $ \left( i,j\right) \in K'$: 
\[ \bigg( {\operatorname{ref}_{\widetilde{\mathfrak A} \to \mathfrak B}  \left( \tilde a_1, \tilde a_3\right)}\bigg)_{ \left( i,j\right)} =  \bigg( {\operatorname{ref}_{\mathfrak A \to \mathfrak B}  \left( a_2, a_4\right)}\bigg)_{ \left( i,j\right)}  \operatorname{mod} im \left( d_0\right)\]
We get the following result: 
\[\tilde a_1= -a_2 +  \left( 1-ST^{-1}\right) a_4, \hspace*{0.2cm} \tilde a_3 = -ST^{-1} a_4\]
and consequently the map $N_{\pi}$ is given by the matrix
\[
	\begin{pmatrix}
		-1 & 1-ST^{-1}\\
		0 & -ST^{-1}
	\end{pmatrix}
\]
For calculating $S_{\pi}^0: \mathcal L_{\pi} \to \mathcal L_0$, we will use exactly the same procedure, except that we have to take care about the continuation to the vector space $\mathbb V$. \\
First we fix another vector space $\mathbb W:= \mathcal K^0_c$ where $c =  \left( 0, 0, \frac 1 2, 0\right) \in \overline A \times \{ 0\}$ and consider the following diagram:  

	\[
	\begin{xy}
		\xymatrixrowsep{1.0cm}
		\xymatrixcolsep{0.7cm}
		\xymatrix
  		{
  		 \mathcal L_{\pi} \ar[d] _{\cong} \ar[rrr]^{S^0_{\pi}}  &  & & \mathcal L_{0} \ar[d]^{\cong}\\
		 H^1 \left( \overline A \times \{ \pi\}, \overline{\beta}^{\pi}_! \mathcal K^{\pi} \right) \ar[d] _{\Gamma_{\pi}} \ar[rrr] &  & &H^1 \left( \overline A \times \{ 0\}, \overline{\beta}^{0}_! 		\mathcal K^{0}\right) 																								\ar[ld]_{\widetilde{\Gamma}		_{0}}  \ar[d] ^{\Gamma_{0}} \\
		 \mathcal K^{\pi}_{x_2} \oplus  \mathcal K^{\pi}_{x_4} \ar[rr]^{\mu_{\pi}^{0}} &&\mathcal K^{0}_{x_2} \oplus  \mathcal K^{0}_{x_4} \ar[r]^{\nu_{0}} \ar[d] &\mathcal K^{0}																									_{x_1} \oplus  	\mathcal 	K^{0}_{x_3} \ar[d] \\
		  && \mathbb W \oplus \mathbb W \ar[r]^{N_{0}} & \mathbb W \oplus \mathbb W
 	 	}
	\end{xy}
	\]
\vspace*{0.2cm}
	
As before, for the determination of the map $\nu_{0}$  (respectively $N_{0}$) we combine the coverings $\mathfrak A$ and $\widetilde {\mathfrak A}$ of $\overline A \times \{0\}$ to the refined covering $\mathfrak B$  and get the refinement maps ${\operatorname{ref}}_{\mathfrak A \to \mathfrak B}$ and ${\operatorname{ref}}_{\widetilde{\mathfrak A} \to \mathfrak B}$ which induce isomorphisms on the cohomology groups. \\
We get:
\[
N_0=
	\begin{pmatrix}
		-TS^{-1} & 0\\
		1-TS^{-1} & -1
	\end{pmatrix}
\]
Now we extend $\nu_0$ respectively $N_0$ to the vector space $\mathbb V \oplus \mathbb V$ by $\mu_{\pi}^{0}$ (which does not affect $N_0$):
	\[
	\begin{xy}

		\xymatrix
		{
		\mathcal K^{\pi}_{x_2} \oplus \mathcal K^{\pi}_{x_4} \ar[r]^{\mu_{\pi}^0} \ar[d] & \mathcal K^{0}_{x_2} \oplus \mathcal K^{0}_{x_4} \ar[r]^{\nu_0} \ar[d] & \mathcal K^{0}_{x_1} 		\oplus \mathcal K^{0}_{x_3} \ar[d] & \mathcal K^{\pi}_{x_1} \oplus \mathcal K^{\pi}_{x_3} \ar[l]_{\mu_{\pi}^0} \ar[d]\\
		\mathbb V \oplus \mathbb V \ar[r]^{Id} & \mathbb W \oplus \mathbb W \ar[r]^{N_0} & \mathbb W \oplus \mathbb W& \mathbb V \oplus \mathbb V  \ar[l]_{Id}
		}
	\end{xy}
	\]
\vspace*{0.2cm}

Let us fix two isomorphisms $\Sigma_0$ and $\Sigma_{\pi}$, which we will call the standard identification of  $H^1 \left( \overline A \times \{ 0\}, \overline{\beta}^0_! \mathcal K^0\right)$, respectively $ H^1 \left( \overline A \times \{ \pi\}, \overline{\beta}^{\pi}_! \mathcal K^{\pi}\right)$ with the vector space  $\mathbb V \oplus \mathbb V$.
\[ \Sigma_0:  H^1 \left( \overline A \times \{ 0\}, \overline{\beta}^0_! \mathcal K^0\right) \xrightarrow{\Gamma_0} \mathcal K_{x_1}^0 \oplus \mathcal K_{x_3}^0 \xrightarrow{\mu_0^{\pi}}  \mathcal K_{x_1}^{\pi} \oplus \mathcal K_{x_3}^{\pi} \rightarrow \mathbb V \oplus \mathbb V \]
\[ \Sigma_{\pi}:  H^1 \left( \overline A \times \{ \pi\}, \overline{\beta}^{\pi}_! \mathcal K^{\pi}\right) \xrightarrow {\Gamma_{\pi}} \mathcal K_{x_2}^{\pi} \oplus \mathcal K_{x_4}^{\pi} \rightarrow \mathbb V \oplus \mathbb V \]

Summarizing the previous calculations, we can now prove Theorem \ref{theorem2}:


\subsection {Conclusion: Proof of Theorem \ref{theorem2}}
  
At first remark that  $\mu_0^{\pi} \circ \mu_{\pi}^0: \mathcal K^{\pi}_x \oplus  \mathcal K^{\pi}_y \to  \mathcal K^{0}_x \oplus  \mathcal K^{0}_y \to  \mathcal K^{\pi}_x \oplus  \mathcal K^{\pi}_y$ is the isomorphism arising from varying the angel $\vartheta$ via the path $\gamma_U: [0,1] \to \mathbb S^1, \tau \mapsto \pi + \tau \cdot 2 \pi$. This corresponds to the monodromy $U$ around the divisor component  $\{0\} \times \mathbb P^1$.\\
Furthermore from Theorem \ref{Thm1} we know that  
\[ \left( H^1 \left( \overline A \times \{ 0\}, \beta^0_! \mathcal K^0\right), H^1 \left( \overline A \times \{ \pi\}, \beta^{\pi}_1 \mathcal K^{\pi}\right), \sigma_0^{\pi}, \sigma_{\pi}^0\right)\]
 defines a set of Stokes data. With the standard identifications of our vector spaces we get the following diagram: 
	\[
	\begin{xy} 
		\xymatrixrowsep{0.7cm}
		\xymatrix
		{	
		\mathbb V \oplus \mathbb V \ar[rr]^{N_{\pi}}&&\mathbb V \oplus \mathbb V\\
		\mathcal K_{x_1}^{\pi} \oplus \mathcal K_{x_3}^{\pi} \ar[u]&& \\
		\mathcal K_{x_1}^{0} \oplus \mathcal K_{x_3}^{0} \ar[u]_{\mu_0^{\pi}}&& \mathcal K_{x_2}^{\pi} \oplus \mathcal K_{x_4}^{\pi} \ar[uu] \\
		H^1 \left( \overline A \times \{ 0\}, \overline{\beta}^0_! \mathcal K^0\right)  \ar@ /^-1.5cm/[ddd]_{\Sigma_0} \ar@ /^1.5cm/[uuu]^{\Sigma_0}  \ar[d]^{\Gamma_0} 				\ar[u]_{\Gamma_0} \ar@<2pt>[rr]^ {\sigma_0^{\pi}}&& H^1 \left( \overline A \times \{ \pi\}, \overline{\beta}^{\pi}_1 \mathcal K^{\pi}\right)  \ar@ /^1.5cm/ [ddd]^{\Sigma_{\pi}}  		\ar@ /^-1.5cm/ [uuu]_{\Sigma_{\pi}}  \ar[d]_{\Gamma_{\pi}} \ar[u]^{\Gamma_{\pi}} \ar@<2pt>[ll]^{\sigma_{\pi}^0}\\
		\mathcal K_{x_1}^{0} \oplus \mathcal K_{x_3}^{0} \ar[d]^{\mu_0^{\pi}}&& \mathcal K_{x_2}^{\pi} \oplus \mathcal K_{x_4}^{\pi} \ar[dd]\\
		\mathcal K_{x_1}^{\pi} \oplus \mathcal K_{x_3}^{\pi} \ar[d]&& \\
		\mathbb V \oplus \mathbb V & \mathbb V \oplus \mathbb V \ar[l]^{\left( \begin{smallmatrix} U &0\\ 0&U \end{smallmatrix}\right)}&  \mathbb V \oplus \mathbb V  \ar[l]^{N_{0}} \\
		}
	\end{xy}
	\]
 Since $\Sigma_0$ and  $\Sigma_{\pi}$ respect the given filtrations of  the vector spaces  $H^1 \left( \overline A \times \{ 0\}, \overline{\beta}^0_! \mathcal K^0\right)$ and $H^1 \left( \overline A \times \{ \pi\}, \overline{\beta}^{\pi}_1 \mathcal K^{\pi}\right)$, it follows that the induced filtrations on $\mathbb V \oplus \mathbb V$ are mutually opposite with respect to  $S_0^1= \Sigma_{\pi} \circ \sigma_0^{\pi} \circ {\Sigma_0}^{-1}$ and $S_1^0= \Sigma_{0} \circ \sigma^0_{\pi} \circ {\Sigma_{\pi}}^{-1}$.  Thus we conclude that $ \left( L_0, L_1, S_0^1, S_1^0\right)$ defines a set of Stokes data for $\mathcal H^0 p_+ \left( \mathcal M \otimes \mathcal E^{\frac 1 y}\right)$.
 
	 \begin{center}
		\noindent\rule{4cm}{0.4pt}
	 \end{center}
 
 \vspace*{1cm}

\emph{Acknowledgements:} I want to thank Marco Hien for various helpful discussions.


\bibliographystyle{alpha}

 \vspace*{0.5cm}

\begin{tabular}{p{7,5cm}l}
&\textsc{Hedwig Heizinger}  \\
&\emph {Lehrstuhl für Algebra und Zahlentheorie}\\
&\emph{Universitätsstraße 14}\\
&\emph{D-86195 Augsburg}\\
&\emph{Email: hedwig.heizinger@math.uni-augsburg.de}\\
\end{tabular}

\end{document}